\documentclass[a4paper,12pt,reqno]{amsart}
\usepackage{amsmath}
\usepackage{amsfonts}
\usepackage{amssymb}
\usepackage{amsthm}
\usepackage{amscd}
\usepackage{color}
\usepackage{verbatim}
\usepackage{eucal,url,amssymb,stmaryrd,enumerate,amscd,}
\usepackage{amsfonts}

\usepackage{amsmath,amsthm,amssymb,amscd,enumerate,eucal,url,stmaryrd}
\usepackage{verbatim}

\newcommand{\R}{\mathbb{R}}
\newcommand{\lie}[1]{\mathfrak{#1}}     

\newcommand{\C}{\mathbb{C}}

\newcommand{\Gtwo}{\mathrm{G}_2}

\newcommand{\GL}{\mathrm{GL}}

\newcommand{\G}{\mathrm{G}_2}

\theoremstyle{plain}
\newtheorem{proposition}{Proposition}

\newtheorem{theorem}[proposition]{Theorem}
\newtheorem{lemma}[proposition]{Lemma}
\newtheorem{corollary}[proposition]{Corollary}
\theoremstyle{definition}
\newtheorem{definition}[proposition]{Definition}
\newtheorem{example}[proposition]{Example}
\theoremstyle{remark}
\newtheorem{remark}{Remark}

\newcommand{\Span}[1]{\operatorname{Span}\left\{#1\right\}}

\begin{document}

\title{Coclosed $\Gtwo$-structures inducing nilsolitons}
\author{Leonardo Bagaglini, Marisa Fern\'andez and Anna Fino}
\maketitle

\begin{abstract}
We show obstructions to the existence of a coclosed $\Gtwo$-structure
on a Lie algebra $\lie{g}$ of dimension seven with non-trivial center.
In particular, we prove that if there exist  a Lie algebra epimorphism 
from $\lie{g}$ to a six-dimensional Lie algebra $\lie{h}$, with kernel contained in the center of $\lie{g}$, then any 
coclosed $\Gtwo$-structure on $\lie{g}$ induces a closed and stable three form on $\lie{h}$ that defines an almost complex structure
on $\lie{h}$. As a consequence, we obtain a classification of the
2-step nilpotent Lie algebras which carry coclosed $\Gtwo$-structures.
We also prove that each one of these Lie algebras has a  coclosed $\Gtwo$-structure inducing a nilsoliton metric,
but this is not true
for 3-step nilpotent Lie algebras with  coclosed $\Gtwo$-structures.  The existence of contact metric structures is also studied.
\end{abstract}

\section{Introduction}\label{sec:intro}

A $7$-dimensional smooth manifold $M$ admits a $\Gtwo$-structure if there is a reduction of the structure group 
of its frame bundle from $\mathrm{Gl}(7,\R)$ to the  exceptional Lie group $\Gtwo$, which can actually be viewed naturally as 
a subgroup of $\mathrm{SO}(7)$. 
Therefore, a $\Gtwo$-structure determines a Riemannian metric and an orientation on $M$. 
The presence of a $\Gtwo$-structure is equivalent to the existence of 
a positive $3$-form, where
the positivity is a natural nondegeneracy condition (see section \ref{algebraic} for details).  Such a 3-form 
$\varphi$ defines a unique Riemannian metric $g_{\varphi}$ and an orientation on $M$. 

 Whenever this 3-form $\varphi$ is covariantly constant
with respect to the Levi-Civita connection of $g_{\varphi}$ or, equivalently, the intrinsic torsion
of the $\Gtwo$-structure vanishes \cite{Salamon}, then the holonomy group is contained
in $\Gtwo$, and this happens if and only if the $3$-form $\varphi$ is closed and coclosed \cite{FernandezGray}.
A $\Gtwo$-structure is called  {\it closed} if the $3$-form  $\varphi$ is closed, and a $\Gtwo$-structure is said to be {\it coclosed} if the $3$-form  $\varphi$ is coclosed.
Usually these two classes  of $\Gtwo$-structures are very different in nature, 
being the closed  condition much more restrictive; for example, coclosed $\Gtwo$-structures always exist on closed spin manifolds and 
satisfy the parametric $h-$principle (\cite{CN}).

In \cite{CF}, Conti and the second author classified the nilmanifolds endowed with an invariant closed $\Gtwo$-structure. 
By a nilmanifold $M$ we mean a compact manifold which is a quotient $M=\Gamma \backslash G$, where $G$ is a connected, 
simply conneced and nilpotent Lie group,
and $\Gamma\subset\,G$ is a lattice. By  Mal'cev Theorem \cite{Malcev}, a lattice $\Gamma\subset\,G$ exists if and only if the
the Lie algebra $\lie{g}$ of $G$ has a basis such that the structure constants of $\lie{g}$ are rational numbers.
Therefore, any nilmanifold is parallelizable, and so spin for any  Riemannian metric.
Then, by Theorem 1.8 from \cite{CN}, every nilmanifold  has a coclosed $\G$-structure.

In this paper we study the existence of invariant coclosed $\Gtwo$-structures on 2-step nilmanifolds  $\Gamma \backslash G$.     
Since invariant  differential forms on $\Gamma \backslash G$ are uniquely determined by forms on  the Lie algebra $\lie{g}$ of $G$,  
we can restrict our attention to seven dimensional nilpotent  Lie algebras classified in \cite{Gong}. We show two necessary 
conditions that a Lie algebra must satisfy if it supports a coclosed $\Gtwo$-structure (see section \ref{sec:sec3}, Corollary \ref{obs1} and Lemma \ref{obs3}).
It turns out that if the seven dimensional nilpotent Lie algebra  $\frak g$ is 2-step decomposable, then it admits a 
coclosed $\Gtwo$-structure (section \ref{sec:sec4}, Corollary \ref{cor:2-stepcoc}), but if $\frak g$ is indecomposable, up to isomorphism, only  $7$  of the $9$  
indecomposable 2-step nilpotent Lie algebras  carry a coclosed $\Gtwo$-structure
(section \ref{sec:sec5}, Theorem \ref{indecomposable-cocalibrated}).  As a result we obtain that  there exist $2$-step nilmanifolds 
which  admit a coclosed $\G$-structure, but not an invariant one.

Nilpotent Lie groups cannot admit  left invariant Einstein metrics. Natural generalizations of Einstein metrics are given by Ricci solitons, which
have been introduced by Hamilton in \cite{Hamilton}.  All known examples of nontrivial homogeneous Ricci solitons are left invariant
metrics on simply connected solvable Lie groups, whose Ricci operator satisfies
the condition  $Ric(g) = \lambda I + D,$
for some $\lambda \in \R$ and some derivation $D$ of the corresponding Lie algebra. The left
invariant metrics satisfying the previous condition are called nilsolitons if the Lie
groups are nilpotent \cite{La1}. Not all nilpotent Lie groups admit nilsoliton metrics,
but if a nilsoliton exists, then it is unique up to automorphism and scaling \cite{La1}.

Closed $\Gtwo$-structures inducing nilsolitons have been studied in \cite{FFM}.  A natural question is thus to see  how is restrictive to impose that 
a  left-invariant coclosed $\Gtwo$-structure on a nilpotent Lie group induces a nilsoliton metric.  In section \ref{sect-nilsoliton} we prove that  all  
the $2$-step nilpotent Lie groups  
admitting a left-invariant coclosed $\Gtwo$-structure have  a coclosed $\Gtwo$-structure inducing a nilsoliton
(see Theorem \ref{nilsoliton-indecomp} and Theorem \ref{nilsoliton-decomposable}). However, this property is not true for higher steps.
Indeed, in Example \ref{example:1}, we show that there exists  a $3$-step nilpotent Lie group  admitting a nilsoliton and a left-invariant coclosed $\Gtwo$-structure,
 but not having any coclosed $\Gtwo$-structure inducing the nilsoliton.

Given a smooth manifold M of dimension seven, one can ask the existence
not only of  a $\Gtwo$-structure  but also of a contact structure. By \cite{ACS} every manifold with $\Gtwo$-structure admits an almost contact structure and two  types of compatibility between  contact and $\Gtwo$-structures have been studied.  
 In section \ref{contact}, we show that if a $2$-step nilmanifold  $\Gamma \backslash G$ admits a contact metric  
$(g, \eta)$ such that   the metric $g$  is induced by a  coclosed  $\Gtwo$-structure, then $G$  is isomorphic to the  $7$-dimensional  
Heisenberg Lie group (Proposition \ref{prop:contact}),
 and therefore  $\Gamma \backslash G$  has an invariant  Sasakian structure.
 In higher step,  this property is not true. Indeed, in Example \ref{example:2},
 we construct a $3$-step (not Sasakian)  $7$-dimensional nilpotent Lie algebra  admitting  a $K$-contact metric structure $(\eta, g)$ such that the 
  metric $g$  is determined by a  coclosed  $\Gtwo$-structure.

{\it Acknowledgements. }
We are grateful to Diego Conti for useful comments. 
The f{}irst  and third author are partially supported by INdAM. The second author is partially supported through Project MICINN (Spain) MTM2014-54804-P and 
Basque Government Project IT: 1094-16. 

\section{Algebraic preliminaries on stable forms}\label{algebraic}

In this section, we collect some results about stable forms on 
an n-dimensional real vector space. We focus on the existence of stable forms in dimension $n=6$ and $n=7$ 
\cite{CLSS, Fei, Hitchin, 
Rei, Schulte}.

Let $V$ be a real vector space of dimension $n$. Consider the representation of the general linear group ${\GL}(V)$ on 
the space $\Lambda^k(V^*)$ of $k$-forms on $V$. An element $\rho\in\Lambda^k(V^*)$ is said to be \emph{stable} if its orbit 
under ${\GL}(V)$ is open in $\Lambda^k(V^*)$. 

In the following proposition, we recall the values of $k$ and $n$, for which there exist open orbits in $\Lambda^k(V^*)$ under  the action of ${\GL}(V)$,
and so for which  there exist stable $k$-forms on $V$. 

\begin{proposition}[\cite{CLSS,Hitchin:StableForms}] 
Let $V$ be an $n$-dimensional real vector space. 
For $1\leq k\leq [\frac{n}{2}]$, the general linear group ${\GL}(V)$ has an open orbit in $\Lambda^k(V^*)$
if and only if $k\leq 2$, or if $k=3$ and $n\,=\,6,\,7$ or $8$.
Moreover, the number of open orbits in $\Lambda^k(V^*)$ is finite.
\end{proposition}

Note that any non-zero 1-form on $V$ is stable. In fact, if $\alpha$ is such a 1-form, then
the orbit of $\alpha$ is an open in $V^*$ since $\GL(V).\,\alpha=V^*\setminus\left\{0\right\}$.
Moreover, according with the following result due to Hictchin (\cite{Hitchin, Hitchin:StableForms}), it happens that for any stable $k$-form $\rho$ on $V$, there is 
a \emph{dual} $(n-k)$-form $\widehat{\rho}\in\Lambda^{n-k}(V^*)$, which is also stable and such that
$\widehat{\rho}\wedge\rho$ defines a volume form on $V$.
 
\begin{proposition}[\cite{CLSS,Hitchin, Hitchin:StableForms, Schulte}]\label{vol-dual}
Let $V$ be an $n$-dimensional oriented vector space, and assume that $k\in\left\{2,n-2\right\}$ and $n$ even, 
or $k\in\left\{3,n-3\right\}$ and $n=6, 7$ or $8$. Then, there exists a ${\GL}(V)$-equivariant map 
\begin{equation}\label{vol}
\varepsilon:\Lambda^k(V^*)\rightarrow \Lambda^n(V^*),
\end{equation}
homogeneous of degree $\frac{n}{k}$, such that it assigns a volume form to a stable $k$-form and it vanishes on non-stable forms.
Given a stable $k$-form $\rho$, the derivative of $\varepsilon$ in $\rho$ defines a  \emph{dual} $(n-k)$-form $\widehat{\rho}\in\Lambda^{n-k}(V^*)$ by the following property
$$d\varepsilon_\rho(\alpha)=\widehat{\rho}\wedge \alpha,$$
for all $\alpha\in\Lambda^k(V^*)$.
Moreover, the dual form $\widehat{\rho}$ is also stable, the identity component of its stabiliser  is equal to  the stabiliser of $\rho$; and the 
forms $\rho$, $\widehat{\rho}$  and the volume form $\varepsilon(\rho)$
are related by the relation
$$\widehat{\rho}\wedge\rho=\frac{n}{k}\varepsilon(\rho).$$
\end{proposition}

In the following subsections, we recall the explicit description of the spaces of the open orbits in $\Lambda^k(V^*)$ when $n=6$ and when $n=7$.

\subsection{Stable forms in dimension six}\label{sec:sixdim}

Let $V$ be a 6-dimensional oriented vector space. 
\begin{theorem}[\cite{Hitchin, Hitchin:StableForms, Hitchin:Clay}] 
There is a unique ${\GL}(V)$ open orbit in $\Lambda^2(V^*)$, which can be characterized as follows 
$$\Lambda_0(V^*)=\left\{\omega\in\Lambda^2(V^*)\;|\;\omega^3\neq 0\right\}.$$
Therefore, if $\omega\in\Lambda_0(V^*)$, then its stabiliser is isomorphic to $\mathrm{Sp}(6,\R)$, its volume form $\varepsilon(\omega)$ 
$($where $\varepsilon$ is the map \eqref{vol}$)$ can 
be chosen to be equal to $\frac{1}{3!}\,\omega^3$, and its dual form $\widehat{\omega}$ $($in the sense of Proposition \ref{vol-dual}$)$
is equal to $\frac{1}{2}\omega^2$.
Moreover, there exists a suitable coframe $\{f^1,\dots,f^6\}$ of $V^*$ such that 
\begin{align*}
&\omega=f^{12}+f^{34}+f^{56},\\&\widehat{\omega}=f^{1234}+f^{1256}+f^{3456},\\&\varepsilon(\omega)=f^{123456},
\end{align*}
where $f^{12}$ stands for $f^{1}\wedge f^{2}$, and so on.
\end{theorem}

In order to describe the open orbits in $\Lambda^3(V^*)$ we proceed as follows. For any $\rho\in\Lambda^3(V^*)$, we consider the map 
$k_\rho:V \mapsto \Lambda^5(V^*)$ defined by
$$k_{\rho}(x)\,=\,\iota_x\rho\wedge\rho,$$
where $\iota_x$ denotes the contraction by the vector $x\in V$.
Clearly $\Lambda^5(V^*)\cong V\otimes\Lambda^6(V^*)$. Indeed, the map $\mu: \Lambda^5(V^*) \mapsto {V\otimes\Lambda^6(V^*)}$ given by
$\mu(\xi)= x\otimes\alpha\in V \otimes\Lambda^6(V^*)$, with $\iota_{x}(\alpha)=\xi$, is an isomorphism. Thus, 
we have the linear map 
\begin{gather}\label{def:capital-K}
K_\rho\,=\,\mu\circ k_{\rho}:V \mapsto  V\otimes\Lambda^6(V^*).
\end{gather}
Then, one can also define the quadratic function 
$$
\lambda: \Lambda^3(V^*) \mapsto \left({\Lambda^6(V^*)}\right)^{\otimes 2}
$$
on the space $\Lambda^3(V^*)$ by
\begin{equation}\label{def:lambda}
6\,\lambda(\rho)=\mathrm{trace}({K_{\rho}}^2)\in\left({\Lambda^6V^*}\right)^{\otimes 2}.
\end{equation}

An element $\lambda(\rho)\in\left({\Lambda^6(V^*)}\right)^{\otimes 2}$, where $\rho\in\Lambda^3(V^*)$, is said to be {\em positive}, 
 and we write $\lambda(\rho)>0$, 
 if $\lambda(\rho)=\nu\otimes\nu$ with $\nu\in\Lambda^6(V^*)$. In this case, the form $\rho$ is called {\em positive}.
An element  $\lambda(\rho)\in\left({\Lambda^6(V^*)}\right)^{\otimes 2}$ is said to be {\em negative}, 
  and we write $\lambda(\rho)<0$, if $-\lambda(\rho)$ is positive, that is $-\lambda(\rho)=\nu\otimes\nu$.
Then, the form $\rho$ is called {\em negative}. 

Now, let us suppose that $\rho\in\Lambda^3(V^*)$ is such that $\lambda(\rho)\not=0$. Then, we can 
consider the linear map
$J_{\rho}: V\to V$ such that, for $x\in V$, $J_{\rho}(x)$ is defined by
\begin{equation}\label{paracomplex-1}
J_{\rho}(x)\,=\,\frac{1}{\sqrt{\lambda(\rho)}}K_{\rho}(x)\in V,
\end{equation} 
if $\lambda(\rho)>0$, and 
\begin{equation}\label{complex-1}
J_{\rho}(x)\,=\,\frac{1}{\sqrt{-\lambda(\rho)}}K_{\rho}(x)\in V,
\end{equation}
if $\lambda(\rho)<0$.

\begin{theorem}[\cite{Hitchin, Hitchin:StableForms, Hitchin:Clay}] \label{stable:6-dim}
 The unique ${\GL}(V)$ open orbits in $\Lambda^3(V^*)$ are the two following sets 
$$\Lambda_+(V^*)=\left\{\rho\in\Lambda^3 V^*\,|\,\lambda(\rho)>0\right\}\qquad \text{and}\qquad\Lambda_{-}(V^*)=\left\{\rho\in\Lambda^3 V^*\,|\,\lambda(\rho)<0\right\},$$ 
where $\lambda(\rho)$ is given by \eqref{def:lambda}. If 
$\rho\in\Lambda_+(V^*)$, 
then the identity component of its stabiliser is isomorphic to $\mathrm{SL}(3,\R)\times\mathrm{SL}(3,\R)$, and there is a coframe $\left\{f^1,\dots,f^6\right\}$ 
of $V^*$ such that $\rho$ has the following expression
$$\rho=f^{123}+f^{456}.$$
If $\rho\in\Lambda_-(V^*)$, the identity component of its stabiliser is isomorphic to $\mathrm{SL}(3,\C)$ and there is a coframe $\left\{f^1,\dots,f^6\right\}$ of $V^*$ 
such that 
$$\rho=-f^{246}+f^{136}+f^{145}+f^{235}.$$
Moreover, $\rho\in\Lambda_+(V^*)$ if and only if the map $J_{\rho}$, defined by $\eqref{paracomplex-1}$, is a paracomplex structure on $V$; 
and $\rho\in\Lambda_-(V^*)$ if and only if the map $J_{\rho}$, defined by $\eqref{complex-1}$, is a complex structure on $V$.
In both cases  the dual form $\widehat{\rho}$ of $\rho$ is given by $\widehat{\rho}=-J_{\rho}^*\rho$.
\end{theorem}

\begin{remark} \label{rem:stable3-forms}
Note that a 3-form $\rho$ on a 6-dimensional oriented vector space $V$ is stable if and only  if $\lambda(\rho)\not=0$.
Moreover, if $\rho\in\Lambda^3(V^*)$ is a 3-form on $V$,  then $k_\rho=k_{-\rho}$, so 
$K_\rho=K_{-\rho}$ and $\lambda(\rho)=\lambda(-\rho)$. Thus, if $\rho\in\Lambda_{+}(V^*)$, then $-\rho\in\Lambda_{+}(V^*)$;
and if $\rho\in\Lambda_{-}(V^*)$, then $-\rho\in\Lambda_{-}(V^*)$.
\end{remark}
\bigskip

\subsection{$\mathrm{SU}(3)$-structures as pairs of stable forms in $\Lambda_{0}(V^*)\times\Lambda_{-}(V^*)$}\label{sec:SU(3)-structures}

We recall the notion of $\mathrm{SU}(3)$-structure on a vector space $V$ of (real) dimension 6. 
An \emph{$\mathrm{SU}(3)$-structure} on $V$ is 
a triple $(g, J, \psi)$ such that $(g, J)$ is an almost Hermitian structure on $V$, and 
$\psi=\psi_++i\,\psi_-$ is a complex $(3,0)$-form, which satisfies
\begin{equation}\label{eq:pair normalized}
\psi_+\wedge \psi_-=\frac{2}{3}\omega^3,
\end{equation}
where $\omega$ is the K\"ahler form of $(g, J)$, and $\psi_+$ and $\psi_-$ are the real part and the imaginary part of $\psi$, respectively.
It is clear that $\omega\wedge\psi_+=\omega\wedge\psi_-=0$ and $\psi_-=J\psi_+$.

\begin{theorem}[\cite{Hitchin}, \cite{Schulte}] \label{th:SU(3)-pairs}
Let $(\omega,\psi_-)\in\Lambda_0(V^*)\times\Lambda_-(V^*)$ such that
$$
\omega\wedge\psi_-=0.
$$
Let $J_{\psi_-}$ be the complex structure on $V$ defined 
by $\eqref{complex-1}$, and let
$h: {V \times V} \to \mathbb{R}$ be the map given by
$$
h(x,y)=\omega(x,J_{\psi_-}y),
$$
for $x, y \in V$. Then, if $h$ is positive definite, 
the stabiliser of the pair $(\omega,\psi_-)$ is a subgroup of $\mathrm{SO}(V,h)$ isomorphic to $\mathrm{SU}(3)$, that is
the pair $(\omega,\psi_-)$ defines an $\mathrm{SU}(3)$-structure for which $J_{\psi_-}$ is the complex structure, $\psi=-J_{\psi_-}^*\psi_- + i \psi_-$ is 
the complex volume form and $h$ the underlying Hermitian metric. Furthermore, any other $\mathrm{SU}(3)$-structure is obtained in this way.\par
If $(\omega,\psi_-)$ is {\em normalized}, that is the condition \eqref{eq:pair normalized} is satisfied, 
then the dual form  $\widehat{\omega}$ and the real part $\psi_+$ of $\psi$ are given by
\begin{align*}
&\widehat{\omega}=\star_h \omega,\\
&\psi_+=\star_h\psi_-,
\end{align*}
where $\star_{h}$ is the Hodge star operator of $h$; and there exists a suitable $($h$-orthonormal)$ 
coframe $\left\{f^1,\dots,f^6\right\}$ of $V^*$ such that 
\begin{equation}\label{ortoh}
\begin{array}{l}
\omega=f^{12}+f^{34}+f^{56},\\
\psi_-=-f^{246}+f^{136}+f^{145}+f^{235}.
\end{array}
\end{equation}
\end{theorem}

\subsection{Stable forms in dimension seven} \label{subject:stable-7-dim}
Let $V$ be a 7-dimensional oriented vector space. For each $3$-form $\varphi\in\Lambda^3(V^*)$ on $V$, we can define the symmetric quadratic form
 $b_\varphi: V\times V \to \Lambda^7(V^*)$ by
$$
6\,b_\varphi(x,y)=\iota_x\varphi\wedge\iota_y\varphi\wedge\varphi,
$$
where $x, y\in V$. If $\varphi$ is stable, then (see \cite{Hitchin, Schulte}) the volume form $\varepsilon(\varphi)$ defined by the map \eqref{vol} is given by 
$$
\varepsilon(\varphi)=\sqrt[9]{\mathrm{det}(b_\varphi)}.
$$
Then, we can consider the symmetric map $g_{\varphi}: V\times V \to {\mathbb{R}}$ given by
\begin{equation}\label{eq:metric7-dim}
g_{\varphi}(x,y)\;\varepsilon(\varphi)=b_{\varphi}(x,y),
\end{equation}
where $x, y\in V$.

\begin{definition} \label{def:G2positive}
Let $V$ be a 7-dimensional oriented vector space. A 3-form $\varphi$ on $V$ is called {\em positive} if $g_\varphi$ is positive definite; and
$\varphi$ is said to be {\em negative} if $g_\varphi$ is indefinite.\par\medskip
\end{definition}

The following result states that the positive and negative forms on $V$ are the unique stable 3-forms on $V$.
\begin{theorem}[\cite{Hitchin}, \cite{Schulte}]\label{stable-7dim}
There are exactly two ${\GL}(V)$ open orbits in $\Lambda^3(V^*)$ which are defined by
$$
\Pi_+(V^*)=\left\{\varphi\in\Lambda^3(V^*) \;|\;g_\varphi\; \text{is positive definite}\right\},
$$
and
$$
\Pi_-(V^*)=\left\{\varphi\in\Lambda^3(V^*) \;|\;g_\varphi  \;\text{is indefinite}\right\}.
$$
If $\varphi\in\Pi_+(V^*)$, then  the stabiliser of $\varphi$ is a subgroup of $\mathrm{O}(V,g_\varphi)$ isomorphic to $\Gtwo$, 
the dual form $\phi=\widehat{\varphi}$ of $\varphi$ is given by
$\phi=\star_{\varphi}\varphi$, where $\star_{\varphi}$ is the Hodge star operator of the metric $g_\varphi$. Moreover, in a suitable $g_\varphi$-orthonormal coframe $\left\{f^1,\dots,f^7\right\}$ of $V^*$, the forms
$\varphi$ and $\phi=\widehat{\varphi}$ have the following expressions
\begin{equation}\label{ortog}
\begin{array}{l}
\varphi=f^{127}+f^{347}+f^{567}+f^{135}-f^{146}-f^{236}-f^{245},\\
\phi=f^{1234}+f^{1256}+f^{1367}+f^{1457}+f^{2357}-f^{2467}+f^{3456},
\end{array}
\end{equation}
If $\varphi\in\Pi_-(V^*)$, then the metric $g_\varphi$ has signature $(3,4)$, and the stabiliser of $\varphi$ is the non-compact group $\Gtwo^*$. In this case,
there exists a $g_\varphi$-orthonormal coframe $\left\{f^1,\dots,f^7\right\}$ of $V^*$ such that the 
forms $\varphi$ and $\phi=\widehat{\varphi}$ are given by 
$$
\begin{array}{l}
\varphi=- f^{127}+f^{347}+f^{567}+f^{135}-f^{146}-f^{236}-f^{245},\\
{ \phi= \star_{\varphi} \varphi =  f^{3456}-f^{1234}-f^{1256}-f^{2467}+f^{1367}+f^{1457}+f^{2357}.}
\end{array}
$$
\end{theorem}\par\medskip

\begin{remark}
The correspondence $\left\{\varphi\mapsto\star_{\varphi}\varphi\right\}$ from the set of   stable $3-$forms to  the one is stable $4-$forms  is $2:1$, preserves both positivity and negativity, and verifies $g_{\varphi}=g_{\star_{\varphi}\varphi}$. A section of this map is completely determined by an orientation on $V$ (see \cite{Bryant}).
\end{remark}

\section{Coclosed $\Gtwo$-structures and obstructions}\label{sec:sec3}
In this section we show
obstructions to the existence of a coclosed  $\Gtwo$ form on a Lie algebra with non-trivial center, 
but not necessarily nilpotent. First, we recall some definitions and results about $\Gtwo$-structures.

A $7$-dimensional smooth manifold $M$ is said to admit a $\Gtwo$-{\em structure} if there is a reduction of the structure group of 
its frame bundle from ${\GL}(7,\mathbb{R})$ to the  exceptional Lie group $\Gtwo$, which can actually be viewed naturally as a subgroup of $\mathrm{SO}(7)$. 
Thus, a $\Gtwo$-structure determines a Riemannian metric and an orientation on $M$. In fact, the presence of a $\Gtwo$-structure is equivalent to the existence of a $3$-form $\varphi$ (the $\Gtwo$ form) on $M$, which is 
\emph{positive} (in the sense of Definition \ref{def:G2positive}) on the tangent space $T_{p} M$ of $M$ at every point $p\in M$.

If $\varphi$ is a $\Gtwo$-form on $M$, then by \eqref{eq:metric7-dim}
$\varphi$ induces both an orientation and a Riemannian metric $g_{\varphi}$ on $M$
given by
\begin{equation}\label{metric}
6\, g_{\varphi} (X,Y)\, vol= \iota_{X}\varphi \wedge \iota_{Y}\varphi \wedge \varphi,
\end{equation}
for any vector fields $X, Y$ on $M$, where $vol$ is the volume form on $M$, and  $\iota_{X}$ denotes the contraction by $X$.
Let $\star_{\varphi}$ be the Hodge star operator determined by $g_{\varphi}$ and the orientation induced by $\varphi$. We will always write 
$\phi$ to mean the dual 4-form of a $\Gtwo$ form, that is
$$
\phi=\star_{\varphi}\varphi.
$$
We say that a manifold $M$ has a {\em coclosed $\Gtwo$-structure}
if there is a $\Gtwo$-structure on $M$ such that the $\Gtwo$ form $\varphi$ is coclosed, that is 
$d\phi=0$.

Now, let $G$ be a $7$-dimensional simply connected 
Lie group with Lie algebra $\lie{g}$. Then, a $\Gtwo$-structure 
on $G$ is \emph{left invariant} if and only if the corresponding
$3$-form is left invariant. According to Theorem \ref{stable-7dim}, a left invariant $\Gtwo$-structure on 
$G$ is defined by a positive 3-form $\varphi\in \Pi_{+}({\lie{g}}^*)$,  
which can be written as 
\begin{equation}\label{eqn:3-forma G2}
 \varphi=e^{127}+e^{347}+e^{567}+e^{135}-e^{146} -e^{236}-e^{245},
\end{equation}
with respect to some orthonormal coframe $\{e^1,\dotsc, e^7\}$ of ${\mathfrak{g}}^*$.
So the dual form $\phi=\star_{\varphi}\varphi$
has the following expression
\begin{equation}\label{eqn:4-forma G2}
 \phi=e^{1234}+e^{1256}+e^{1367}+e^{1457}+e^{2357}-e^{2467}+e^{3456},
\end{equation}
where $e^{127}$ stands for $e^1\wedge e^2\wedge e^7$, and so on.

A $\Gtwo$-structure on $\lie{g}$ is said to be {\em coclosed} (or {\em cocalibrated}) if $\varphi$ is coclosed, that is 
$$
d\phi=0,
$$
where $d$ denotes the Chevalley-Eilenberg differential on ${\lie{g}}^*$. 
By \cite{Malcev} we know that if $\lie{g}$ is nilpotent with rational structure 
constants, then the associated simply connected nilpotent  Lie group $G$ admits a uniform discrete 
subgroup $\Gamma$. Therefore, a $\Gtwo$-structure on $\lie{g}$ determines 
a $\Gtwo$-structure on the compact manifold $\Gamma\backslash G$, which is called a 
compact nilmanifold; and if $\lie{g}$ has a coclosed
$\Gtwo$-structure, the $\Gtwo$-structure on $\Gamma\backslash G$ is also coclosed.

In order to show obstructions to the existence of a coclosed $\Gtwo$-structure on a Lie algebra
$\lie{g}$, let us consider first the case  that $\lie{g}$ is a direct sum of two ideals $\lie{h}$ and $\R$,
$$
\lie{g}=\lie{h}\oplus\R,
$$ 
where $\lie{h}$ is a 6-dimensional Lie algebra. If $\phi$ is a 4-form defining a $\Gtwo$-structure on $\lie{g}$, 
and the  decomposition $\lie{g}=\lie{h}\oplus\R$ is {\em orthogonal with respect to the underlying metric on $\lie{g}$}, then
$$
\phi=\frac{1}{2}\omega^2+\psi_{-}\wedge dt,
$$
where the pair $(\omega, \psi_{-})$ defines an $\mathrm{SU}(3)$-structure on $\lie{h}$, and $t$ is a coordinate on $\mathbb{R}$. 
Now the condition that $\phi$ is closed is equivalent to both $\omega^2$ and $\psi_{-}$ are closed. 
This means that the $\mathrm{SU}(3)$-structure is {\em half-flat}. There are exactly 
24 nilpotent Lie algebras of dimension six that admit a half-flat structure ~\cite{Conti:HalfFlat}. Hence, if we focus our attention 
on decomposable nilpotent Lie algebras, there are at least 24 nilpotent Lie algebras, of dimension seven, 
with a coclosed $\Gtwo$-structure. In Theorem \ref{decomposable-cocalibrated}, we show that exactly
those are the decomposable nilpotent Lie algebras admitting coclosed $\Gtwo$-structures.

Let $\lie{g}$ be a $7$-dimensional Lie algebra with non-trivial center. Then, if $X\in\lie{g}$ belongs to  the center of $\lie{g}$, the quotient $\lie{h}={\lie{g}}/{\Span{X}}$ has a unique Lie algebra structure that makes the projection map
$\pi: \lie{g}\to {\lie{h}}$
a Lie algebra morphism. Thus, 
we have the following short exact sequence of Lie algebras
\begin{equation}\label{fibration}
\begin{CD}
0@>>>\mathbb{R} X@>>>\lie{g}@>\pi>>\lie{h}@>>>0.
\end{CD}
\end{equation}
Moreover, if we assume that $\lie{g}$ is nilpotent, every epimorphism $\lie{g}\to\lie{h}$, with $\lie{h}$ of dimension six, is of this form.

We need also to recall the following result due to Schulte \cite [Proposition 4.5]{Schulte}.
If $\varphi$  is a $\Gtwo$-structure on a $7$-dimensional Lie algebra and we choose 
a vector $X \in \frak g$, of length one with respect to the metric $g_{\varphi}$ determined by $\varphi$, 
then on the orthogonal 
complement  of  the span of  $X$  one has an $\mathrm{SU}(3)$-structure $(\omega, \psi_{-})$ given by 
the 2-form $\omega = \iota_X \varphi$ and the 3-form $\psi_{-} =- \iota_X \phi$,
where $\phi$ is the dual 4-form of $\varphi$, that is $\phi=\star_{\varphi}\varphi$.

\begin{proposition}\label{prop1}
Let $\lie{g}$ be a seven dimensional Lie algebra, with non-trivial center and a $3$-form $\varphi$ defining a 
coclosed $\Gtwo$-structure on $\lie{g}$.
If $\pi: {\lie{g}} \to {\lie{h}}$ is an epimorphism from $\lie{g}$ to a six dimensional 
Lie algebra $\lie{h}$, with kernel contained in the center of  $\lie{g}$,
then $\varphi$ determines an $\mathrm{SU}(3)$-structure $(\omega, \psi_-)$ on $\lie{h}$ such that the 3-form $\psi_-$ is closed.  
\end{proposition}

\begin{proof}
Denote by $g_{\varphi}$ the underlying metric on $\lie{g}$ defined by the $\Gtwo$-structure $\varphi$, and denote by
$\phi$ the dual 4-form of $\varphi$, that is $\phi=\star_{\varphi}\varphi$. 
Let $X$ be a unit vector in the center of $\lie{g}$, and let $\eta=\iota_{X}(g_{\varphi})$ be
the dual form of $X$ with respect to $g_{\varphi}$. By the aforementioned result of \cite[Proposition 4.5]{Schulte}
we know that $\varphi$ induces an $\mathrm{SU}(3)$-structure $(\widetilde{\omega}, \widetilde{\psi_{-}})$ on the orthogonal 
complement  $V$ of  the span of  $X$, that is on $\ker(\eta)$, such that 
$$
\phi=\widetilde{\sigma} + \widetilde{\psi_{-}}\wedge\eta,
$$
where $\widetilde{\sigma}\in\Lambda^4(V^*)$ is the dual 4-form of $\widetilde{\omega}$ with respect to the metric defined by
$(\widetilde{\omega}, \widetilde{\psi_{-}})$ on the space $\ker(\eta)$.

Now, consider the Lie algebra $\lie{h}={\lie{g}}/{\Span{X}}$ and the projection map 
$$\pi: {\lie{g}} \to {\lie{h}}.$$
Clearly  $\lie{h}$ and $V$ are isomorphic as vector spaces. Then, fixed an isomorphism between these spaces and doing the pullback of
the $\mathrm{SU}(3)$-structure $(\widetilde{\omega}, \widetilde{\psi_{-}})$ on $V$, we have an 
$\mathrm{SU}(3)$-structure $(\omega, \psi_{-})$ on $\lie{h}$ such that
\begin{equation}\label{relation: forms}
\phi=\pi^*\sigma + \pi^*\psi_{-}\wedge\eta,
\end{equation}
where $\widetilde{\sigma}\in\Lambda^4(\lie{h}^*)$ is the dual 4-form of $\omega$ with respect to the metric defined by
$(\omega,\psi_{-})$ on $\lie{h}$. Thus,
$$\psi_{-}=\pi_{*}(-\iota_{X}\phi).$$
Moreover, we see that 
$$
d\psi_{-}=0, \qquad d(\pi^*\sigma)\,=\,\pi^*\psi_{-}\wedge d\eta.
$$
In fact, since $d$ commutes with the pullback, from \eqref{relation: forms} we have
$$
0=d\phi=\pi^*(d\sigma)+ \pi^*(d\psi_{-})\wedge \eta - \pi^*\psi_{-}\wedge d\eta.
$$
Taking the contraction by $X$, and using that $X$ is in the center of $\lie{g}$, that is ${\mathrm{ad}}(X)=0$,  
we have $d\psi_{-}=0$ and $d(\pi^*\sigma)\,=\,\pi^*\psi_{-}\wedge d\eta$.
\end{proof}

As a consequence of the previous Proposition, we have the following obstruction to existence of coclosed $\Gtwo$-structures on Lie algebras
with non-trivial center.
\begin{corollary}\label{obs1}
Let $\lie{g}$ be a seven dimensional oriented Lie algebra with non-trivial center. If there is an element $X$ in the center of $\lie{g}$, such that
\begin{equation}\label{condition:obst}
\pi_{*}(\iota_{X}\kappa) \not\in\Lambda_{-}(\lie{h}^*), 
\end{equation}
for every closed 4-form $\kappa\in\Lambda^{4}(\lie{g}^*)$, where $\pi: \lie{g} \to \lie{h}={\lie{g}}/{\Span{X}}$ is the projection map,
 then $\lie{g}$ does not admit any coclosed $\Gtwo$-structure.
\end{corollary}
\begin{proof}
Let $X$ be a non-zero vector in the center of $\lie{g}$ such that the condition \eqref{condition:obst} is satisfied, for any closed 
4-form on $\lie{g}$. Suppose that there is a 4-form $\phi$ on $\lie{g}$ defining a coclosed $\Gtwo$-structure. Then, by Proposition \ref{prop1},
the 4-form $\phi$ determines an $\mathrm{SU}(3)$-structure $(\omega, \psi_-)$ on $\lie{h}={\lie{g}}/{\Span{X}}$, where $\psi_{-}=\pi_{*}(-\iota_{X}\phi)$
by \eqref{relation: forms}.  
Now, Remark \ref{rem:stable3-forms} and Theorem \ref{th:SU(3)-pairs} imply that $\pi_{*}(\iota_{X}\phi)\in\Lambda_{-}(\lie{h}^*)$, which contradicts  the condition \eqref{condition:obst}.
So, $\lie{g}$ does not admit coclosed $\Gtwo$-structures.
\end{proof}

Another obstruction to the existence of a coclosed $\Gtwo$-structure on a Lie algebra is given by the following lemma.
\begin{lemma}\label{obs3}
Let $\lie{g}$ be a $7$-dimensional Lie algebra. If there are non-zero vectors $X$ and $Y$ in $\lie{g}$ such that $(\iota_{X}\iota_{Y}\kappa)^2=0$, for every closed $4$-form 
$\kappa$ on $\lie{g}$, then $\lie{g}$ does not admit coclosed $\Gtwo$-structures.
\end{lemma}
\begin{proof}
It follows from \eqref{eqn:4-forma G2}.
\end{proof}

The following result, whilst straightforward, 
turns out effective to show that some Lie algebras do not admit coclosed $\Gtwo$-structures
(see Proposition \ref{no-cocalibrated2}).
\begin{lemma}\label{lem2}
Let $(h,J)$ be an 
almost Hermitian structure on a six dimensional oriented vector space $V$, with orthogonal complex structure $J$, Hermitian metric $h$ and fundamental two-form $\omega(.,.)=h(J.,.)$.
Then, for any $J$-invariant $4-$dimensional subspace $W$ of $V$,  we have that the restriction to $W$ of the 4-form $(\star\omega)$
is non-zero, that is
$(\star\omega)|_W\neq 0$,
where $\star$ denotes the Hodge $\star$-operator of the metric $h$.
\end{lemma}
\begin{proof}
Since $W$ is $J$-invariant, there exist non-zero vectors $x,y\in W$ such that $\left\{x,Jx,y,Jy\right\}$ is an orthonormal basis of the space $W$ (with respect to $h$).
Let $z\in V$ be a unit vector orthogonal to the space $W$. Then $\left\{x,Jx,y,Jy,z,Jz\right\}$ is a (real) $h$-orthonormal basis of $V$. Therefore,
\begin{eqnarray*}
(\star\omega)(x\wedge Jx\wedge y\wedge Jy)=\omega(z\wedge Jz)=||z||^2=1,
\end{eqnarray*}
which proves that $(\star\omega)|_W\neq 0$. 
\end{proof}

\section{Decomposable case}\label{sec:sec4}
In this section we classify the seven dimensional decomposable nilpotent Lie algebras which admit coclosed $\Gtwo$-structures.
Recall that a Lie algebra is called {\em decomposable} if it is the direct sum of two ideals.

For convenience, from now on we will use the following notation.
Suppose that $\lie{g}$ is a $7$-dimensional Lie algebra whose dual space $\lie{g}^*$ is spanned by $\{ e^1,\ldots ,e^7\}$ satisfying
$$
de^i=0, \quad 1\leq i\leq 4, \qquad de^5=e^{23}, \qquad de^6=e^{34}, \qquad de^7=e^{36}. 
$$
Then we will write
$$
\lie{g}\,=\,(0,0,0,0,23,34,36) 
$$
with the same meaning. Moreover, we will denote by $\{e_1,\ldots ,e_7\}$ the basis of $\lie{g}$ dual to  $\{e^1,\ldots ,e^7\}$.

\begin{proposition}\label{no-cocalibrated1}
If $\lie{g}$ is one of the following seven dimensional nilpotent Lie algebras
\begin{align*}
&\lie{g}_1=(0,0,0,0,12,15,0), \quad\quad\,\,\,\,\,\,\,\,  \lie{g}_2=(0,0,0,0,23,34,36), \quad\quad\,\,\,\,    \lie{g}_3=(0,0,0,12,13,14,0), \\
&\lie{g}_4=(0,0,0,12,14,24,0), \quad\quad\,\,\,\,\, \lie{g}_5=(0,0,12,13,14,23+15,0), \,\lie{g}_6=(0,0,12,13,14,15,0), \\
&\lie{g}_7=(0,0,12,13,14,34-25,0), \,\lie{g}_8=(0,0,12,13,14+23,34-25,0),
\end{align*}
then $\lie{g}$ carries no coclosed $\Gtwo$-structures.
\end{proposition}
\begin{proof}
 Using Lemma \ref{obs3}, we will prove that each Lie algebra  $\lie{g}_s$, $s\in\left\{1,\dots,8\right\}$, listed in the statement,
it does not admit any coclosed $\Gtwo$-structure.
For this, we will show that there are non-zero vectors $X_s, Y_{s}\in\lie{g}_s$ such that $\Big(\iota_{X_{s}}(\iota_{Y_{s}}\kappa_{s})\Big)^2=0$,
for any closed $4$-form $\kappa_{s}$ on  $\lie{g}_s$.
 \paragraph{${\bf s=1}$}  A generic closed 4-form $\kappa_1$ on $\lie{g}_1=(0,0,0,0,12,15,0)$ has the following expression
\begin{equation*}
\begin{aligned}
\kappa_1=&c_{1234} e^{1234}+c_{1235} e^{1235}+c_{1236} e^{1236}+c_{1237} e^{1237}+c_{1245} e^{1245}+c_{1246} e^{1246}+c_{1247} e^{1247}\\
&+c_{1256} e^{1256}+c_{1257} e^{1257}+c_{1267} e^{1267}+c_{1345} e^{1345}+c_{1346} e^{1346}+c_{1347} e^{1347}+c_{1356} e^{1356}\\
&+c_{1357} e^{1357}+c_{1367} e^{1367}+c_{1456} e^{1456}+c_{1457} e^{1457}+c_{1467} e^{1467}+c_{1567} e^{1567}+c_{2345} e^{2345}\\
&+c_{2347} e^{2347}+c_{2356} e^{2356}+c_{2357} e^{2357}+c_{2456} e^{2456}+c_{2457} e^{2457}+c_{2567} e^{2567},
\end{aligned}
\end{equation*}
where $c_{ijkl}$ are arbitrary real numbers. Now one can check that if the coefficient $c_{2356}$ of $\kappa_1$ vanishes, then
$$
 \Big(\iota_{e_3}(\iota_{e_6}\kappa_1)\Big)^2=0,
$$
that is $\lie{g}_1$ satisfies the hypothesis of Lemma \ref{obs3} for $X_{1}=e_3$ and $Y_{1}=e_6$.
If $c_{2456}=0$, then $$ \Big(\iota_{e_4}(\iota_{e_6}\kappa_1)\Big)^2=0.$$ But if 
$c_{2356}$ and $c_{2456}$ are both non-zero, then for $X_{1}=c_{2356}\,e_4-c_{2456}\,e_3$ and $Y_{1}=e_6$,  we have
$$ \Big(\iota_{c_{2 3 5 6}\,e_4-c_{2 4 5 6}\,e_3}(\iota_{e_6}\kappa_1)\Big)^2=0.$$
 \paragraph{${\bf s=2}$}  A generic closed 4-form $\kappa_2$ on $\lie{g}_2=(0,0,0,0,23,34,36)$ has the following expression
\begin{equation*}
\begin{aligned}
\kappa_2=&c_{1234} e^{1234}+c_{1235} e^{1235}+c_{1236} e^{1236}+c_{1237} e^{1237}+c_{1245} e^{1245}+c_{1246} e^{1246}+c_{1247} e^{1247}\\
&+c_{1345} e^{1345}+c_{1346} e^{1346}+c_{1347} e^{1347}+c_{1356} e^{1356}+c_{1357} e^{1357}+c_{1367} e^{1367}\\
&-c_{1247} e^{1456}+c_{1467} e^{1467}+c_{2345} e^{2345}+c_{2346} e^{2346}+c_{2347} e^{2347}+c_{2356} e^{2356}\\
&+c_{2357} e^{2357}+c_{2367} e^{2367}+c_{2456} e^{2456}+c_{2467} e^{2467}+c_{3456} e^{3456}+c_{3457} e^{3457}+c_{3467} e^{3467},\\
&+c_{3567} e^{3567},
\end{aligned}
\end{equation*}
where $c_{ijkl}$ are arbitrary real numbers. Taking $X_2=e_5$ and $Y_2=e_7$, we have $\Big(\iota_{e_5}(\iota_{e_7}\kappa_2)\Big)^2=0$.

 \paragraph{${\bf s=3}$}  A generic closed 4-form $\kappa_3$ on $\lie{g}_3=(0,0,0,12,13,14,0)$ is expressed as follows
 \begin{equation*}
\begin{aligned}
\kappa_3=&c_{1234} e^{1234}+c_{1235} e^{1235}+c_{1236} e^{1236}+c_{1237} e^{1237}+c_{1245} e^{1245}+c_{1246} e^{1246}+c_{1247} e^{1247}\\
&+c_{1256} e^{1256}+c_{1257} e^{1257}+c_{1267} e^{1267}+c_{1345} e^{1345}+c_{1346} e^{1346}+c_{1347} e^{1347}+c_{1356} e^{1356}\\
&+c_{1357} e^{1357}+c_{1367} e^{1367}+c_{1456} e^{1456}+c_{1457} e^{1457}+c_{1467} e^{1467}+c_{1567} e^{1567}\\
&+c_{2345} e^{2345}+c_{2346} e^{2346}+c_{2347} e^{2347}+c_{2357} e^{2357}+c_{2367} (e^{2367}+e^{2457})+c_{2467} e^{2467}.
\end{aligned}
\end{equation*}
where $c_{ijkl}$ are arbitrary real numbers. For $X_3=e_5$ and $Y_3=e_6$, we have $ \Big(\iota_{e_5}(\iota_{e_6}\kappa_3)\Big)^2=0$.
 \paragraph{${\bf s=4}$} A generic closed 4-form $\kappa_4$ on $\lie{g}_4=(0,0,0,12,14,24,0)$ is expressed as follows
 \begin{equation*}
\begin{aligned}
\kappa_4=&c_{1234} e^{1234}+c_{1235} e^{1235}+c_{1236} e^{1236}+c_{1237} e^{1237}+c_{1245} e^{1245}+c_{1246} e^{1246}+c_{1247} e^{1247}\\
&+c_{1256} e^{1256}+c_{1257} e^{1257}+c_{1267} e^{1267}+c_{1345} e^{1345}+c_{1346} e^{1346}+c_{1347} e^{1347}+c_{1357} e^{1357}\\
&+c_{1367} e^{1367}+c_{1456} e^{1456}+c_{1457} e^{1457}+c_{1467} e^{1467}+c_{2345} e^{2345}+c_{2346} e^{2346}+c_{2347} e^{2347}\\
&+c_{1367} e^{2357}+c_{2367} e^{2367}+c_{2456} e^{2456}+c_{2457} e^{2457}+c_{2467} e^{2467}+c_{3567} e^{3567},
\end{aligned}
\end{equation*}
where $c_{ijkl}$ are arbitrary real numbers. Then, for $X_4=e_5$ and $Y_4=e_6$, we have $\Big(\iota_{e_5}(\iota_{e_6}\kappa_4)\Big)^2=0$.

 \paragraph{${\bf s=5}$} 
 A generic closed 4-form $\kappa_5$ on $\lie{g}_5=(0,0,12,13,14,23+15,0)$ has the following expression
\begin{equation*}
\begin{aligned}
\kappa_5=&c_{1234} e^{1234}+c_{1235} e^{1235}+c_{1236} e^{1236}+c_{1237} e^{1237}+c_{1245} e^{1245}+c_{1246} e^{1246}+c_{1247} e^{1247}\\
&+c_{1256} e^{1256}+c_{1257} e^{1257}+c_{1267} e^{1267}+c_{1345} e^{1345}+c_{1346} e^{1346}+c_{1347} e^{1347}+c_{1356} e^{1356}\\
&+c_{1357} e^{1357}+c_{1367} e^{1367}+c_{1456} e^{1456}+c_{1457} e^{1457}+c_{1467} e^{1467}+c_{1567} e^{1567}+c_{2345} e^{2345}\\
&+c_{1456} e^{2346}+c_{2347} e^{2347}-c_{1467} e^{2357}+c_{2367} e^{2367}-(c_{1567}+c_{2367}) e^{2457},
\end{aligned}
\end{equation*}
where $c_{ijkl}$ are arbitrary real numbers. Then, for $X_5=e_5$ and $Y_5=e_6$, we have $\Big(\iota_{e_5}(\iota_{e_6}\kappa_5)\Big)^2=0$.
 \paragraph{${\bf s=6}$} 
 A generic closed 4-form $\kappa_6$ on $\lie{g}_6=(0,0,12,13,14,15,0)$ has the following expression
\begin{equation*}
\begin{aligned}
\kappa_6=&c_{1234} e^{1234}+c_{1235} e^{1235}+c_{1236} e^{1236}+c_{1237} e^{1237}+c_{1245} e^{1245}+c_{1246} e^{1246}+c_{1247} e^{1247}\\
&+c_{1256} e^{1256}+c_{1257} e^{1257}+c_{1267} e^{1267}+c_{1345} e^{1345}+c_{1346} e^{1346}+c_{1347} e^{1347}+c_{1356} e^{1356}\\
&+c_{1357} e^{1357}+c_{1367} e^{1367}+c_{1456} e^{1456}+c_{1457} e^{1457}+c_{1467} e^{1467}+c_{1567} e^{1567}+c_{2345} e^{2345}\\
&+c_{2347} e^{2347}+c_{2367} e^{2367} -c_{2367} e^{2457},\\
\end{aligned}
\end{equation*}
where $c_{ijkl}$ are arbitrary real numbers. So, for $X_6=e_5$ and $Y_6=e_7$, $\Big(\iota_{e_5}(\iota_{e_7}\kappa_6)\Big)^2=0$.
 \paragraph{${\bf s=7}$}
 A generic closed 4-form $\kappa_6$ on $\lie{g}_7=(0,0,12,13,14,34-25,0)$ has the following expression
\begin{equation*}
\begin{aligned}
\kappa_7=&c_{1234} e^{1234}+c_{1235} e^{1235}+c_{1236} e^{1236}+c_{1237} e^{1237}+c_{1245} e^{1245}+c_{1246} e^{1246}+c_{1247} e^{1247}\\
&+c_{1256} e^{1256}+c_{1257} e^{1257}+c_{1267} e^{1267}+c_{1345} e^{1345}+c_{1256} e^{1346}+c_{1347} e^{1347}+c_{1356} e^{1356}\\
&+c_{1357} e^{1357}+c_{1367} e^{1367}+c_{1456} e^{1456}+c_{1457} e^{1457}+c_{1467} e^{1467}+c_{2345} e^{2345}+c_{2346} e^{2346}\\
&+c_{2347} e^{2347}-c_{1267} e^{2357}+c_{2367} e^{2367}-c_{1367} e^{2457}-c_{1467} e^{3457},
\end{aligned}
\end{equation*}
where $c_{ijkl}$ are arbitrary real numbers. Thus, $(i_{e_5}(i_{e_7}\kappa_7))^2=0$. 
\smallskip
 \paragraph{${\bf s=8}$} 
 A generic closed 4-form $\kappa_6$ on $\lie{g}_8=(0,0,12,13,14+23,34-25,0)$ has the following expression
\begin{equation*}
\begin{aligned}
\kappa_8=&c_{1234} e^{1234}+c_{1235} e^{1235}+c_{1236} e^{1236}+c_{1237} e^{1237}+c_{1245} e^{1245}+c_{1246} e^{1246}+c_{1247} e^{1247}\\
&+c_{1256} e^{1256}+c_{1257} e^{1257}+c_{1267} e^{1267}+c_{1345} e^{1345}+c_{1256} e^{1346}+c_{1347} e^{1347}+c_{1356} e^{1356}\\
&+c_{1357} e^{1357}+c_{1367} e^{1367}+c_{1456} e^{1456}+c_{1457} e^{1457}+c_{1467} e^{1467}+c_{2345} e^{2345}\\
&+c_{2346} e^{2346}+c_{2347} e^{2347}-c_{1456} e^{2356}-(c_{1267}+c_{1457})e^{2357}+c_{2367} e^{2367}-c_{1367} e^{2457}\\
&-c_{1467} e^{3457},
\end{aligned}
\end{equation*}
where $c_{ijkl}$ are arbitrary real numbers. On can check that if the coefficient $c_{1456}$ vanishes, then
$$
 \Big(\iota_{e_5}(\iota_{e_6}\kappa_8)\Big)^2=0,
$$
that is $\lie{g}_8$ satisfies the hypothesis of Lemma \ref{obs3} for $X_{8}=e_5$ and $Y_{8}=e_6$.
If $c_{2346}=0$, then $$ \Big(\iota_{e_4}(\iota_{e_6}\kappa_8)\Big)^2=0.$$ But if 
$c_{1456}$ and $c_{2346}$ are both non-zero, then for $X_{8}=c_{2346}\,e_5+c_{1456}\,e_4$ and $Y_{8}=e_6$,  we have
$$ \Big(\iota_{c_{2346}\,e_5+c_{1456}\,e_4}(\iota_{e_6}\kappa_8)\Big)^2=0.$$
\end{proof}

Moreover, using Lemma \ref{lem2}, we have the following proposition.

\begin{proposition}\label{no-cocalibrated2}
None of the following seven dimensional nilpotent Lie algebras
\begin{align*}
&\lie{l}_1=(0,0,0,12,13-24,14+23,0), \quad \lie{l}_2=(0,0,0,12,14,13-24,0), \\ 
&\lie{l}_3=(0,0,0,12,13+14,24,0)
\end{align*}
admits coclosed $\Gtwo$-structures.  
\end{proposition}
\begin{proof}
We will prove, by contradiction,  that no closed 4-form $\tau_s$ defines a coclosed $\Gtwo$-structure on the Lie algebra $\lie{l}_s$ $(s=1, 2, 3)$. We proceed case by case.
\paragraph*{For $\bf s=1$} A generic closed 4-form $\tau_1$ on $\lie{l}_1=(0,0,0,12,13-24,14+23,0)$ has the following expression 
\begin{equation*}
\begin{aligned}
\tau_1=&c_{1234} e^{1234}+c_{1235} e^{1235}+c_{1236} e^{1236}+c_{1237} e^{1237}+c_{1245} e^{1245}+c_{1246} e^{1246}+c_{1247} e^{1247}\\
&+c_{1256} e^{1256}+c_{1257} e^{1257}+c_{1267} e^{1267}+c_{1345} e^{1345}+c_{1346} e^{1346}+c_{1347} e^{1347}+c_{1356} e^{1356}\\
&+c_{1357} e^{1357}+c_{1367} e^{1367}+c_{1456} e^{1456}+c_{1457} e^{1457}+c_{1467} e^{1467}+c_{2345} e^{2345}+c_{2346} e^{2346}\\
&+c_{2347} e^{2347}-c_{1456} e^{2356}+c_{2357} e^{2357}+c_{2367} e^{2367}-c_{1356} e^{2456}+(c_{1357}+c_{1467}+c_{2367}) e^{2457}\\
&+c_{2467} e^{2467},
\end{aligned}
\end{equation*}
where $c_{ijkl}$ are arbitrary real numbers. Let us suppose that, for some real numbers $c_{ijkl}$, the 4-form $\tau_1$ defines a coclosed $\Gtwo$-structure on $\lie{l}_1$.
Since $e_7$ is in the center of $\lie{l}_1$, Proposition \ref{prop1} implies that 
$$\nu_1=-\pi_*(\iota_{e_7}\tau_1)$$
 is a negative 3-form on the Lie algebra 
$\lie{h}_1=\lie{l}_1/\Span{e_7}$, where $\pi: \lie{l}_1\rightarrow \lie{h}_1$ is the projection. Thus, 
\begin{align*}
\nu_1=&c_{1 2 3 7}e^{123}+c_{1 2 4 7}e^{124}+c_{1 2 5 7}e^{125}+c_{1 2 6 7}e^{126}+c_{1 3 4 7}e^{134}+c_{1 3 5 7}e^{135}+c_{1 3 6 7}e^{136}+c_{1 4 5 7}e^{145}\\
&+c_{1 4 6 7}e^{146}+c_{2 3 4 7}e^{234}+c_{2 3 5 7}e^{235}+c_{2 3 6 7}e^{236}+(c_{1357}+c_{1467}+c_{2367}) e^{245}+c_{2 4 6 7}e^{246}.
\end{align*}
 We claim that the map $K_{\nu_1}$, defined by \eqref{def:capital-K}, has the following expression
$$
K_{\nu_1}=(K^{\nu_1}_{ab})\otimes e^{123456}=\left(\begin{matrix}
K^{\nu_1}_{11}&K^{\nu_1}_{12}&0&0&0&0\\
K^{\nu_1}_{21}&K{\nu_1}_{22}&0&0&0&0\\
K^{\nu_1}_{31}&K^{\nu_1}_{32}&K^{\nu_1}_{33}&K^{\nu_1}_{34}&0&0\\
K^{\nu_1}_{41}&K^{\nu_1}_{42}&K^{\nu_1}_{43}&K^{\nu_1}_{44}&0&0\\
K^{\nu_1}_{51}&K^{\nu_1}_{52}&K^{\nu_1}_{53}&K^{\nu_1}_{54}&K^{\nu_1}_{55}&K^{\nu_1}_{56}\\
K^{\nu_1}_{61}&K^{\nu_1}_{62}&K^{\nu_1}_{63}&K^{\nu_1}_{64}&K^{\nu_1}_{65}&K^{\nu_1}_{66}\\
\end{matrix}\right)\otimes e^{123456},
$$ 
where $a, b\in\{1,\ldots,6\}$, and $K^{\nu_1}_{ab}$ is a polynomial function of the coefficients $c_{ijkl}$ that appear in the expression of $\nu_1$.
In fact, by \eqref{def:capital-K} it turns out that 
$$
(\iota_{e_b}\nu_1)\wedge\nu_1=\sum_{1\leq a \leq 6} K^{\nu_1}_{ab}\,e^{1\dots\hat{a}\dots 6}.
$$
Therefore, $K^{\nu_1}_{ab}\,e^{1\dots 6}=(\iota_{e_b}\nu_1)\wedge\nu_1\wedge e^{a}$. Then, one can check
that $K^{\nu_1}_{ab}=0$, for $a=1,2$ and $b\geq 3$, and  also $K^{\nu_1}_{ab}=0$, for $a=3,4$ and $b=5,6$. Thus, the claim is true.

Since $\nu_1\in\Lambda_{-}({\lie{h}_1}^*)$, \eqref{complex-1} and Theorem \ref{stable:6-dim} imply that $\nu_1$ defines the almost complex structure $J_{\nu_1}$ on 
$\lie{h}_1$ given by
$$J_{\nu_1}=\frac{1}{\sqrt{|-\lambda(\nu_1)|}}\left(\begin{matrix}
K^{\nu_1}_{11}&K^{\nu_1}_{12}&0&0&0&0\\
K^{\nu_1}_{21}&K^{\nu_1}_{22}&0&0&0&0\\
K^{\nu_1}_{31}&K^{\nu_1}_{32}&K^{\nu_1}_{33}&K^{\nu_1}_{34}&0&0\\
K^{\nu_1}_{41}&K^{\nu_1}_{42}&K^{\nu_1}_{43}&K^{\nu_1}_{44}&0&0\\
K^{\nu_1}_{51}&K^{\nu_1}_{52}&K^{\nu_1}_{53}&K^{\nu_1}_{54}&K^{\nu_1}_{55}&K^{\nu_1}_{56}\\
K^{\nu_1}_{61}&K^{\nu_1}_{62}&K^{\nu_1}_{63}&K^{\nu_1}_{64}&K^{\nu_1}_{65}&K^{\nu_1}_{66}\\
\end{matrix}\right).$$ 
Therefore, the subspace $W=\Span{e_3,e_4,e_5,e_6}$ 
is $J_{\nu_1}$-invariant.   Now, consider an 
arbitrary 1-form $\eta=\sum_{r=1}^{7} C_{r}e^r$ on $\lie{l}_1$. According with \eqref{relation: forms}, and taking into account the 
expressions of $\tau_{1}$ and $\nu_1$, we see that the 4-form $\sigma=\pi_*\left(\tau_1-\pi^*\nu_1\wedge\eta\right)$ on $\lie{h}_1$, 
has zero component in $e^{3456}$ vanishes. Hence,
${\sigma}|_W=\sigma(e_3,e_4,e_5,e_6)=0$ contradicting Lemma \ref{lem2}. Thus $\tau_1$ never defines a coclosed $\Gtwo$-structure.

\paragraph*{For $\bf s=2$} A generic closed four-form on $\lie{l}_2=(0,0,0,12,14,13-24,0)$ has the following expression 
\begin{equation*}
\begin{aligned}
\tau_2= &c_{1234} e^{1234}+c_{1235} e^{1235}+c_{1236} e^{1236}+c_{1237} e^{1237}+c_{1245} e^{1245}+c_{1246} e^{1246}+c_{1247} e^{1247}\\
&+c_{1256} e^{1256}+c_{1257} e^{1257}+c_{1267} e^{1267}+c_{1345} e^{1345}+c_{2456} e^{1346}+c_{1347} e^{1347}+c_{2456} e^{1356}\\
&+c_{1357} e^{1357}+ (c_{2467} - c_{2357}) e^{1367}+c_{1456} e^{1456}+c_{1457} e^{1457}+c_{1467} e^{1467}+c_{2345} e^{2345}\\
&+c_{2346} e^{2346}+c_{2347} e^{2347}+c_{2357} e^{2357}+c_{2367} e^{2367}+c_{2456} e^{2456}+c_{2457} e^{2457} +c_{2467} e^{2467},
\end{aligned}
\end{equation*}
where $c_{ijkl}$ are arbitrary real numbers. Let us suppose that, for some real numbers $c_{ijkl}$, the 4-form $\tau_2$ defines a coclosed $\Gtwo$-structure on $\lie{l}_2$. Since $e_7$ is in the center of $\lie{l}_2$,  by Proposition \ref{prop1} we have that  $\nu_2=-\pi_*(i_{e_7}\tau_2)$ is a negative $3$-form on  $ \frak h_2 = \lie{l}_2/\Span{e_7}$ where 
 $\pi: \lie{l}_2\rightarrow \frak h_2$  is the natural projection. Thus, $\nu_2$ is given by 
\begin{align*}
\nu_2=&c_{1 2 3 7}e^{123}+c_{1 2 4 7}e^{124}+c_{1 2 5 7}e^{125}+c_{1 2 6 7}e^{126}+c_{1 3 4 7}e^{134}+c_{1 3 5 7}e^{135}+c_{1 3 6 7}e^{136}+c_{1 4 5 7}e^{145}+\\&c_{1 4 6 7}e^{146}+c_{2 3 4 7}e^{234}+c_{2 3 5 7}e^{235}+c_{2 3 6 7}e^{236}+c_{2 4 5 7}e^{245}+(c_{2 3 5 7}+c_{1 3 6 7})e^{246}.
\end{align*}
We claim that the map $K_{\nu_2}$, defined by \eqref{def:capital-K}, has the following expression
$$K_{\nu_2}=(K^{\nu_2}_{ab})\otimes e^{123456}=\left(\begin{matrix}
K^{\nu_2}_{11}&K^{\nu_2}_{12}&0&0&0&0\\
K^{\nu_2}_{21}&K^{\nu_2}_{22}&0&0&0&0\\
K^{\nu_2}_{31}&K^{\nu_2}_{32}&K^{\nu_2}_{33}&K^{\nu_2}_{34}&0&0\\
K^{\nu_2}_{41}&K^{\nu_2}_{42}&K^{\nu_2}_{43}&K^{\nu_2}_{44}&0&0\\
K^{\nu_2}_{51}&K^{\nu_2}_{52}&K^{\nu_2}_{53}&K^{\nu_2}_{54}&K^{\nu_2}_{55}&K^{\nu_2}_{56}\\
K^{\nu_2}_{61}&K^{\nu_2}_{62}&K^{\nu_2}_{63}&K^{\nu_2}_{64}&K^{\nu_2}_{65}&K^{\nu_2}_{66}\\
\end{matrix}\right)\otimes e^{123456},$$
where $a, b\in\{1,\ldots,6\}$, and $K^{\nu_2}_{ab}$ is a polynomial function of the coefficients $c_{ijkl}$ that appear in the expression of $\nu_2$.
In fact, by \eqref{def:capital-K} it turns out that 
$$
(\iota_{e_b}\nu_2)\wedge\nu_2=\sum_{1\leq a \leq 6} K^{\nu_2}_{ab}\,e^{1\dots\hat{a}\dots 6}.
$$
Therefore, $K^{\nu_2}_{ab}\,e^{1\dots 6}=(\iota_{e_b}\nu_2)\wedge\nu_2\wedge e^{a}$. Then, one can check
that $K^{\nu_2}_{ab}=0$, for $a=1,2$ and $b\geq 3$, and  also $K^{\nu_2}_{ab}=0$, for $a=3,4$ and $b=5,6$. Thus, the claim is true.\par 
Since $\nu_2\in\Lambda_{-}({\lie{h}_2}^*)$, \eqref{complex-1} and Theorem \ref{stable:6-dim} imply that $\nu_2$ defines the almost complex structure $J_{\nu_2}$ on 
$\lie{h}_2$ given by 
$$J_{\nu_2}=\frac{1}{|\sqrt{-\lambda(\nu_2)}|}\left(\begin{matrix}
K^{\nu_2}_{11}&K^{\nu_2}_{12}&0&0&0&0\\
K^{\nu_2}_{21}&K^{\nu_2}_{22}&0&0&0&0\\
K^{\nu_2}_{31}&K^{\nu_2}_{32}&K^{\nu_2}_{33}&K^{\nu_2}_{34}&0&0\\
K^{\nu_2}_{41}&K^{\nu_2}_{42}&K^{\nu_2}_{43}&K^{\nu_2}_{44}&0&0\\
K^{\nu_2}_{51}&K^{\nu_2}_{52}&K^{\nu_2}_{53}&K^{\nu_2}_{54}&K^{\nu_2}_{55}&K^{\nu_2}_{56}\\
K^{\nu_2}_{61}&K^{\nu_2}_{62}&K^{\nu_2}_{63}&K^{\nu_2}_{64}&K^{\nu_2}_{65}&K^{\nu_2}_{66}\\
\end{matrix}\right).$$ 
Therefore, the subspace $W=\Span{e_3,e_4,e_5,e_6}$ 
is $J_{\nu_2}$-invariant.
Now, consider an 
arbitrary 1-form $\eta=\sum_{r=1}^{7} C_{r}e^r$ on $\lie{l}_2$. According with \eqref{relation: forms}, and taking into account the 
expressions of $\tau_{2}$ and $\nu_2$, we see that the 4-form $\sigma=\pi_*\left(\tau_2-\pi^*\nu_2\wedge\eta\right)$ on $\lie{h}_2$, 
has zero component in $e^{3456}$. Hence,
${\sigma}|_W=\sigma(e_3,e_4,e_5,e_6)=0$ contradicting Lemma \ref{lem2}. Thus $\tau_2$ never defines a coclosed $\Gtwo$-structure.

\paragraph*{For $\bf s=3$}
A generic closed four-form on $\lie{l}_3 = (0,0,0,12,13+14,24,0)$ has the following expression 
\begin{equation*}
\begin{aligned}
\tau_2=&c_{1234} e^{1234}+c_{1235} e^{1235}+c_{1236} e^{1236}+c_{1237} e^{1237}+c_{1245} e^{1245}+c_{1246} e^{1246}+c_{1247} e^{1247}\\
&+c_{1256} e^{1256}+c_{1257} e^{1257}+c_{1267} e^{1267}+c_{1345} e^{1345}+c_{1346} e^{1346}+c_{1347} e^{1347}\\
&+c_{1357} e^{1357}+ c_{1367} e^{1367}+c_{1456} e^{1456}+c_{1457} e^{1457}+c_{1467} e^{1467}+c_{2345} e^{2345}+c_{2346} e^{2346}\\
&+c_{2347} e^{2347}+c_{2357} e^{2357}+c_{2367} e^{2367}+c_{2356} e^{2456}+(c_{1367}+c_{2357}) e^{2457} +c_{2467} e^{2467},
\end{aligned}
\end{equation*}
where $c_{ijkl}$ are arbitrary real numbers.  Let us suppose that, for some real numbers $c_{ijkl}$, the 4-form $\tau_3$ defines a coclosed $\Gtwo$-structure on $\lie{l}_2$. Since $e_7$ is in the center of $\lie{l}_1$,  by Proposition \ref{prop1} we have that  $\nu_3=-\pi_*(i_{e_7}\tau_3)$ is a negative $3$-form on  $ \frak h_3 = \lie{l}_3/\Span{e_7}$ where 
 $\pi: \lie{l}_3\rightarrow \frak h_3$  is the natural projection. Thus, $\nu_3$ is given by 
\begin{align*}
\nu_3=&c_{1 2 3 7}e^{123}+c_{1 2 4 7}e^{124}+c_{1 2 5 7}e^{125}+c_{1 2 6 7}e^{126}+c_{1 3 4 7}e^{134}+c{1 3 5 7}e^{135}+c_{1 3 6 7}e^{136}+c_{1 4 5 7}e^{145}+\\&c_{1 4 6 7}e^{146}+c_{2 3 4 7}e^{234}+c_{2 3
 5 7}e^{235}+c_{2 3 6 7}e^{236}+(c_{2 3 5 7}-c_{1 3 6 7})e^{245}+c_{2 4 6 7}e^{246}.
\end{align*}
We claim that the map $K_{\nu_3}$, defined by \eqref{def:capital-K}, has the following expression
$$K_{\nu_3}=(K^{\nu_3}_{ab})\otimes e^{123456}=\left(\begin{matrix}
K^{\nu_3}_{11}&K^{\nu_3}_{12}&0&0&0&0\\
K^{\nu_3}_{21}&K^{\nu_3}_{22}&0&0&0&0\\
K^{\nu_3}_{31}&K^{\nu_3}_{32}&K^{\nu_3}_{33}&K^{\nu_3}_{34}&0&0\\
K^{\nu_3}_{41}&K^{\nu_3}_{42}&K^{\nu_3}_{43}&K^{\nu_3}_{44}&0&0\\
K^{\nu_3}_{51}&K^{\nu_3}_{52}&K^{\nu_3}_{53}&K^{\nu_3}_{54}&K^{\nu_3}_{55}&K^{\nu_3}_{56}\\
K^{\nu_3}_{61}&K^{\nu_3}_{62}&K^{\nu_3}_{63}&K^{\nu_3}_{64}&K^{\nu_3}_{65}&K^{\nu_3}_{66}\\
\end{matrix}\right)\otimes e^{123456},$$
where $a, b\in\{1,\ldots,6\}$, and $K^{\nu_3}_{ab}$ is a polynomial function of the coefficients $c_{ijkl}$ that appear in the expression of $\nu_3$.
In fact, by \eqref{def:capital-K} it turns out that 
$$
(\iota_{e_b}\nu_3)\wedge\nu_3=\sum_{1\leq a \leq 6} K^{\nu_3}_{ab}\,e^{1\dots\hat{a}\dots 6}.
$$
Therefore, $K^{\nu_3}_{ab}\,e^{1\dots 6}=(\iota_{e_b}\nu_3)\wedge\nu_3\wedge e^{a}$. Then, one can check
that $K^{\nu_3}_{ab}=0$, for $a=1,2$ and $b\geq 3$, and  also $K^{\nu_3}_{ab}=0$, for $a=3,4$ and $b=5,6$. Thus, the claim is true.\par 
Since $\nu_3\in\Lambda_{-}({\lie{h}_3}^*)$, \eqref{complex-1} and Theorem \ref{stable:6-dim} imply that $\nu_3$ defines the almost complex structure $J_{\nu_32}$ on 
$\lie{h}_3$ given by 
$$J_{\nu_3}=\frac{1}{|\sqrt{-\lambda(\nu_3)}|}\left(\begin{matrix}
K^{\nu_3}_{11}&K^{\nu_3}_{12}&0&0&0&0\\
K^{\nu_3}_{21}&K^{\nu_3}_{22}&0&0&0&0\\
K^{\nu_3}_{31}&K^{\nu_3}_{32}&K^{\nu_3}_{33}&K^{\nu_3}_{34}&0&0\\
K^{\nu_3}_{41}&K^{\nu_3}_{42}&K^{\nu_3}_{43}&K^{\nu_3}_{44}&0&0\\
K^{\nu_3}_{51}&K^{\nu_3}_{52}&K^{\nu_3}_{53}&K^{\nu_3}_{54}&K^{\nu_3}_{55}&K^{\nu_3}_{56}\\
K^{\nu_3}_{61}&K^{\nu_3}_{62}&K^{\nu_3}_{63}&K^{\nu_3}_{64}&K^{\nu_3}_{65}&K^{\nu_3}_{66}\\
\end{matrix}\right).$$ 
Therefore, the subspace $W=\Span{e_3,e_4,e_5,e_6}$ 
is $J_{\nu_3}$-invariant.
Now, consider an 
arbitrary 1-form $\eta=\sum_{r=1}^{7} C_{r}e^r$ on $\lie{l}_3$. According with \eqref{relation: forms}, and taking into account the 
expressions of $\tau_{3}$ and $\nu_3$, we see that the 4-form $\sigma=\pi_*\left(\tau_3-\pi^*\nu_3\wedge\eta\right)$ on $\lie{h}_3$, 
has zero component in $e^{3456}$. Hence,
${\sigma}|_W=\sigma(e_3,e_4,e_5,e_6)=0$ contradicting Lemma \ref{lem2}. Thus $\tau_3$ never defines a coclosed $\Gtwo$-structure.
\end{proof}

 By Gong's classification there exist, up to isomorphism, 35 decomposable  $7$-dimensional nilpotent Lie algebras. We will show that 24 of these Lie algebras admit  a coclosed $\Gtwo$-structure.

\begin{theorem}\label{decomposable-cocalibrated}
Among the 35 decomposable nilpotent Lie algebras, of dimension 7, those that have a  coclosed $\Gtwo$-structure 
are 
\begin{align*}
&\lie{n}_1=(0,0,0,0,0,0,0), \qquad \qquad \qquad \quad\,\,\,\, \lie{n}_2=(0,0,0,0,0,12,0),\\
&\lie{n}_3=(0,0,0,0,0,12+34,0),\qquad \qquad \quad \lie{n}_4=(0,0,0,0,12,13,0), \\ 
&\lie{n}_5=(0,0,0,0,12,34,0), \quad \quad \quad\quad\quad \quad\,\,\,\,  \lie{n}_6=(0,0,0,0,13-24,14+23,0), \\
&\lie{n}_7=(0,0,0,0,12,14+23,0),\quad \quad \quad\quad\,\,\,  \lie{n}_8=(0,0,0,0,12,14+25,0),\\
&\lie{n}_9=(0,0,0,0,12,15+34,0), \quad \quad \quad\quad\,\,  \lie{n}_{10}=(0,0,0,12,13,23,0),\\
&\lie{n}_{11}=(0,0,0,12,13,24,0), \qquad \qquad \quad \,\,\,\, \lie{n}_{12}=(0,0,0,12,13,14+23,0),\\
&\lie{n}_{13}=(0,0,0,12,23,14+35,0),  \quad \quad \quad\,\,\,  \lie{n}_{14}=(0,0,0,12,23,14-35,0), \\ 
&\lie{n}_{15}=(0,0,0,12,13,14+35,0), \quad \quad \quad\,\,\,  \lie{n}_{16}=(0,0,0,12,14,15,0), \\
&\lie{n}_{17}=(0,0,0,12,14,15+24,0),\quad \quad \quad\,\,\,  \lie{n}_{18}=(0,0,0,12,14,15+24+23,0),\\
&\lie{n}_{19}=(0,0,0,12,14,15+23,0), \quad \quad \quad\,\,\,   \lie{n}_{20}=(0,0,0,12,14-23,15+34,0), \\
&\lie{n}_{21}=(0,0,12,13,23,14+25,0), \quad \quad \quad  \lie{n}_{22}=(0,0,12,13,23,14-25,0), \\
&\lie{n}_{23}=(0,0,12,13,23,14,0), \quad \quad \quad\quad \quad\, \lie{n}_{24}=(0,0,12,13,14+23,15+24,0).
\end{align*}
\end{theorem}
\begin{proof} 
By Proposition \ref{no-cocalibrated1} and Proposition \ref{no-cocalibrated2}, we know 
that there are 11 decomposable nilpotent Lie algebras not admitting coclosed
$\Gtwo$-structures. This implies that there are at most $24$ decomposable nilpotent Lie algebras
having a  coclosed  $\Gtwo$-structure. 
But all the 24 Lie algebras listed in the statement have such a $\Gtwo$-structure.
It is clear on the abelian Lie algebra $\lie{n}_1=(0,0,0,0,0,0,0)$. 
Moreover, 
every non-abelian Lie algebra  $\lie{n}_s$, $s\in\left\{2,\dots,24\right\}$, listed in the statement,
is a direct sum of two ideals $\lie{h}_s$ and $\R$. In fact, $\lie{n}_s$ is an abelian extension of $\lie{h}_s$,
$$\lie{n}_s=\lie{h}_s\oplus\R e_7,$$
where $\lie{h}_s$ is a 6-dimensional Lie algebra, and $e_7=\frac{\partial}{\partial t}$ with $t$ a coordinate on $\R$.

By \cite{Conti:HalfFlat}, we know that  $\lie{h}_s$ has a half-flat $\mathrm{SU}(3)$-structure $(\omega_s,\psi_{s}^{-})$.
Thus, the 4-form $\phi_s$ on $\lie{n}_s$ given by
$$
\phi_s=\frac{1}{2}\omega_{s}^2+\psi_{s}^{-}\wedge dt,
$$
defines a  coclosed $\Gtwo$-structure on $\lie{n}_s$, which completes the proof.
\end{proof}

\begin{corollary}\label{cor:2-stepcoc}
Any decomposable 2-step nilpotent Lie algebra, of dimension seven, has a  coclosed $\Gtwo$-structure.
\end{corollary}
\begin{proof}
Among the 35 decomposable nilpotent Lie algebras, of dimension seven, those that are 2-step nilpotent  
are the 7 Lie algebras $\lie{n}_s$, where $s\in\{2,3,4,5,6,7,10\}$, defined in Theorem \ref{decomposable-cocalibrated}. 
Thus, all of them carry  coclosed  $\Gtwo$-structures.
\end{proof}

\section{Indecomposable 2-step nilpotent case}\label{sec:sec5}
In this section we complete the classification of seven dimensional 2-step nilpotent Lie algebras which admit a coclosed  $\Gtwo$-structure. 
We have seen that there are exactly 7  decomposable Lie algebras of this type. 
In order to discuss the indecomposable 2-step nilpotent Lie algebras, we refer to Gong's classification \cite{Gong} (see also \cite{CF}). 
This list consists of 131 Lie algebras and 9 one-parameter families. None of the one-parameter families defines a 2-step nilpotent Lie algebra. 
Among the indecomposable Lie algebras (with no parameters)  there are only nine which are 2-step nilpotent Lie algebras. They are the following:
\begin{align*}
17&= \left(0,0,0,0,0,0,12+34+56\right),\\
27A&=\left(0,0,0,0,0,12,14+35\right),\\
27B&=\left(0,0,0,0,0,12+34,15+23\right),\\
37A&=\left(0,0,0,0,12, 23,24\right),\\
37B&=\left(0,0,0,0,12,23,34\right),\\
37B1&=\left(0,0,0,0,12-34,13+24,14\right),\\ 
37C&=\left(0,0,0,0,12+34,23,24\right),\\ 
37D&=\left(0,0,0,0,12+34,13,24\right),\\ 
37D1&=\left(0,0,0,0,12-34,13+24,14-23\right).\\
\end{align*}
\begin{theorem}\label{indecomposable-cocalibrated}
Up to isomorphism, the unique indecomposable 2-step nilpotent Lie algebras  carrying  a  coclosed  $\Gtwo$-structure are the following:
$17, 37A, 37B,  37B1,  37C,  37D$ and $37D1$.
\end{theorem}
\begin{proof}
We know that a $\Gtwo$-structure on a Lie algebra $\lie{g}$ can be defined either  by a 3-form or, equivalently, by a 4-form which have the expression
 given by \eqref{eqn:3-forma G2} or  \eqref{eqn:4-forma G2}, respectively with 
respect to some orthonormal coframe $\{e^1,\dotsc, e^7\}$ of ${\mathfrak{g}}^*$. For each Lie algebra appearing in the statement, 
an appropriate coframe and the corresponding 4-form are given as follows
\begin{itemize}
\item $17$: \,\,\,\, $\phi_{1}=e^{1234}+e^{1256}+e^{1367}+e^{1457}+e^{2357}-e^{2467}+e^{3456}$, \,\,\,\,  $\{e^1,e^2,e^3,e^4,e^5,e^6, e^7\}$;\\
\item$37A$: \,\,$\phi_{2}=e^{1234}+e^{1257}-e^{1356}-e^{1467}-e^{2367}+e^{2456}+e^{3457}$, \,\,\,\,  $\{e^3,e^1,e^2,e^4,e^5,e^6,e^7\}$;\\
\item$37B$: \,\,$\phi_{3}=e^{1234}+e^{1457}+e^{2357}+\frac{\sqrt{2}}{2}(e^{1256}-e^{1356}+e^{1267}+e^{1367}+e^{2456})\\
\qquad +\frac{\sqrt{2}}{2}(e^{3456}-e^{2467}+e^{3467})$, \qquad \qquad \qquad $\{e^1,\frac{\sqrt{2}}{2}(e^2-e^3),\frac{\sqrt{2}}{2}(e^2+e^3),e^4,e^5,
e^6,e^7\}$;\\
\item$37B1$:\,$\phi_{4}=e^{1234}+e^{1267}+e^{1357}+e^{1456}+e^{2356}-e^{2457}+e^{3467}$, \,\,\,\, $\{e^1,e^4,e^2,e^3,e^5,e^6,e^7\}$;\\
\item $37C$: \,\,$\phi_{5}=-e^{1234}-e^{1267}+e^{1357}+e^{1456}-e^{2356}+e^{2457}+e^{3467}$,\,\,\,\, $\{e^2,e^3,e^4,e^1,e^6,e^5,e^7\}$;\\
\item $37D$: \,\,$\phi_{6}=e^{1234}+e^{1267}+e^{3467}+\frac{\sqrt{2}}{2}(e^{1356}+e^{1357}+e^{1456}-e^{1457}+e^{2356})\\
-\frac{\sqrt{2}}{2}(e^{2357}-e^{2456}-e^{2457})$,\qquad \qquad \qquad $\{e^1,\frac{\sqrt{2}}{2}(e^3+e^4), e^2,\frac{\sqrt{2}}{2}(e^3-e^4),e^5,e^6, e^7\}$;\\
\item $37D1$:\,$\phi_{7}=-e^{1234}-e^{1267}-e^{1356}+e^{1457}-e^{2357}-e^{2456}+e^{3467}$,\,\, $\{e^1,e^3,e^2,e^4,e^5,-e^6,e^7\}$.
\end{itemize}
It is straightforward to verify that  each $\phi_{i}$ is closed on the corresponding Lie algebra.

It remains to prove that the Lie algebras $27A$ and $27B$ do not 
admit any  coclosed  $\Gtwo$-structure. 
To that effect, we show that for these Lie algebras the hypothesis of Lemma \ref{obs3} is satisfied for $X=e_6$ and $Y=e_7$.
In fact, let $\alpha$ be a generic closed 4-form on $27A$. Then, one can check that $\alpha$ has the following expression
\begin{equation*}
\begin{aligned}
\alpha=&c_{1234} e^{1234}+c_{1235} e^{1235}+c_{1236} e^{1236}+c_{1237} e^{1237}+c_{1245} e^{1245}+c_{1246} e^{1246}+c_{1247} e^{1247}\\
&+c_{1256} e^{1256}+c_{1257} e^{1257}+c_{1345} e^{1345}+c_{1346} e^{1346}+c_{1347} e^{1347}+c_{1356} e^{1356}\\
&+c_{1357} e^{1357}+c_{1367} e^{1367}+c_{1456} e^{1456}+c_{1457} e^{1457}+c_{1567} e^{1567}+c_{2345} e^{2345}\\
&+c_{2346} e^{2346}+c_{2347} e^{2347}+c_{2356} e^{2356}+c_{2357} e^{2357}+c_{2456} e^{2456}+c_{2457} e^{2457}\\
&+(c_{1247}-c_{2357}) e^{3456}+c_{3457} e^{3457},
\end{aligned}
\end{equation*}
where $c_{ijkl}$ are arbitrary real numbers. Thus, $\Big(\iota_{e_{6}}(\iota_{e_{7}}\alpha)\Big)^2=0$. 
A generic closed 4-form $\beta$ on  the Lie algebra $27B$ is expressed as
\begin{equation*}
\begin{aligned}
\beta=&c_{1234} e^{1234}+c_{1235} e^{1235}+c_{1236} e^{1236}+c_{1237} e^{1237}+c_{1245} e^{1245}+c_{1246} e^{1246}+c_{1247} e^{1247}\\
&+c_{1256} e^{1256}+c_{1257} e^{1257}+c_{1345} e^{1345}+c_{1346} e^{1346}+c_{1347} e^{1347}+c_{1356} e^{1356}\\
&+c_{1357} e^{1357}+c_{1367} e^{1367}+c_{1456} e^{1456}+c_{1457} e^{1457}+c_{2345} e^{2345}\\
&+c_{2346} e^{2346}+c_{2347} e^{2347}+c_{2356} e^{2356}+c_{2357} e^{2357}+c_{2456} e^{2456}+c_{2457} e^{2457}\\
&+(c_{1256}+c_{1457}-c_{2347}) e^{3456}+c_{3457} e^{3457},
\end{aligned}
\end{equation*}
where $c_{ijkl}$ are arbitrary real numbers. Hence, $\Big(\iota_{e_{6}}(\iota_{e_{7}}\beta)\Big)^2=0$. 
\end{proof}
\section{Coclosed $\Gtwo$-structures inducing nilsolitons} \label{sect-nilsoliton}

In this section we prove that any 2-step nilpotent Lie algebra, admitting a  coclosed  $\Gtwo$-structure, also has  a 
nilsoliton inner product determined by a   coclosed  $\Gtwo$-structure. This  result is not true for higher steps. 
Indeed, we give an example of a 
3-step nilpotent Lie algebra  supporting a  nilsoliton and  coclosed  $\Gtwo$-structures but none  coclosed  $\Gtwo$-structure
induces a nilsoliton.

Let $\lie{n}$ be a nilpotent Lie algebra. According to Lauret \cite{Lauret},
an inner product $g$ on $\lie{n}$  is  called \emph{nilsoliton} if its Ricci endomorphism 
$Ric(g)$ differs from a derivation $D$ of 
$\lie{n}$ by a scalar multiple of the identity map $I$, that is  
if there exists a real number $\lambda$ such  that  
\begin{equation} \label{condition-derivation} 
Ric(g)= \lambda I + D.
\end{equation}
Not all nilpotent Lie algebras admit nilsoliton inner products, but if a nilsoliton 
exists, then it is unique up to automorphism and scaling \cite{Lauret}.
\begin{theorem}\label{nilsoliton-indecomp} 
Any indecomposable 2-step nilpotent Lie algebra admitting a   coclosed  $\Gtwo$-structure,  also has a  coclosed $\Gtwo$-structure inducing  a nilsoliton.
\end{theorem}
\begin{proof} 
By Theorem \ref{indecomposable-cocalibrated} we know that, up to isomorphism, the Lie algebras $17, 37A, 37B,\\
 37B1, 37C, 37D$ and  $37D1$ are the unique
indecomposable 2-step nilpotent admitting  coclosed  $\Gtwo$-structures. For each of these Lie algebras,
defined in Section \ref{sec:sec5} in terms of a basis $\{e^1,\cdots, e^7\}$ of the dual space, we consider a new basis $\{f^1,\dots,f^7\}$ 
and a coclosed  $\Gtwo$-structure defined by a 3-form $\varphi_{i}$ $(1\leq i\leq 7)$ which determines the 
inner product such that the basis $\{f^1,\dots,f^7\}$ is orthonormal.
The basis $\{f^j\}$, the $\Gtwo$-structure $\varphi_{i}$ and the appropriate coframe defining $\varphi_{i}$ are given as follows 
\begin{itemize}
\item $17$: \,\,\, $f^i=e^i,  \quad 1\leq i\leq 6, \qquad  f^7=\frac{\sqrt{6}}{6}\,e^7$. \\
Then, $17=\left (0,0,0,0, 0, 0,  \frac{\sqrt{6}}{6} f^{12} +\frac{\sqrt{6}}{6} f^{34}+ \frac{\sqrt{6}}{6} f^{56} \right)$, and\\
$\varphi_{1}=f^{127}+f^{347}+f^{567}+f^{135}-f^{146}-f^{236}-f^{245}$, \,\,\,\,  $\{f^1,f^2,f^3,f^4,f^5,f^6,f^7\}$;\\
\item$37A$: \,\, $f^i=e^i,  \quad 1\leq i\leq 4, \qquad  f^j=\frac{\sqrt{6}}{6}\,e^j,  \quad 5\leq j\leq 7$. \\
Then, $37A=\left(0,0,0,0, \frac{\sqrt{6}}{6} f^{12}, \frac{\sqrt{6}}{6} f^{23},  \frac{\sqrt{6}}{6}f^{24}\right)$, and\\
$\varphi_{2}=-f^{137}+f^{247}+f^{567}-f^{126}-f^{145}-f^{235}-f^{346}$, \,\,\,\,  $\{f^3,f^1,f^2,f^4,f^5,f^6,f^7\}$;\\
\item$37B$: \,\, $f^i=e^i,  \quad 1\leq i\leq 4, \qquad  f^5=\frac{\sqrt{5}}{5}\,e^5,\qquad  f^6=\frac{\sqrt{6}}{6}\,e^6, \qquad  f^7=\frac{\sqrt{5}}{5}\,e^7$. \\
Thus, $37B=\left(0,0,0,0, \frac{\sqrt{5}}{5} f^{12}, \frac{\sqrt{10}}{10} f^{23}, \frac{\sqrt{5}}{5}f^{34}\right)$, and\\
$\varphi_{3}=-f^{146}-f^{236}+f^{567}+ \frac{\sqrt{2}}{2}(f^{123}+f^{125}-f^{135}-f^{137}-f^{245})\\
+ \frac{\sqrt{2}}{2}(f^{247}+f^{345}+f^{347})$,  \,\,\,\,\qquad\qquad   $\{f^1,\frac{\sqrt{2}}{2}(f^2-f^3),\frac{\sqrt{2}}{2}(f^2+f^3),f^4,f^5,f^6,f^7\}$;\\
\item$37B1$: \, $f^i=e^i,  \quad 1\leq i\leq 4, \qquad  f^j=\frac{\sqrt{10}}{10}\,e^j, \quad 5\leq j\leq 7$. \\
So, $37B1=\left(0,0,0,0,\frac{\sqrt{10}}{10}(f^{12} - f^{34}), \frac{\sqrt{10}}{10}(f^{13} + f^{24}),  \frac{\sqrt{10}}{10} f^{14}\right)$, and\\
$\varphi_{4}=f^{125}+f^{345}+f^{567}+ \frac{\sqrt{2}}{2}(-f^{136}+f^{137}+f^{146}+f^{147}+f^{236})\\
+ \frac{\sqrt{2}}{2}(f^{237}+f^{246}-f^{247})$,  \,\,\,\,\qquad\qquad   $\{f^1,\frac{\sqrt{2}}{2}(f^3+f^4), f^2,\frac{\sqrt{2}}{2}(f^3-f^4),f^5,f^6, f^7\}$;\\
\item $37C$: \,\, $f^i=e^i,  \quad 1\leq i\leq 4, \qquad  f^j=\frac{\sqrt{2}}{4}\,e^j, \quad 5\leq j\leq 7$. \\
Thus, $37C=\left(0,0,0,0, \frac{\sqrt{2}}{4}(f^{12} +f^{34}), \frac{\sqrt{2}}{4} f^{23}, \frac{\sqrt{2}}{4} f^{24}\right)$, and\\
$\varphi_{5}=f^{147}+f^{237}+f^{567}+f^{125}-f^{136}+f^{246}+f^{345}$, \,\,\,\,  $\{f^1,f^4,f^2,f^3,f^5,f^6,f^7\}$;\\
\item $37D$: \,\, $f^i=e^i,  \quad 1\leq i\leq 4, \qquad  f^j=\frac{\sqrt{3}}{6}\,e^j, \quad 5\leq j\leq 7$. \\
Then, $37D=\left(0,0,0,0, \frac{\sqrt{3}}{6}(f^{12} +f^{34}, \frac{\sqrt{6}}{6} f^{13}, \frac{\sqrt{6}}{6} f^{24}\right)$, and\\
$\varphi_{6}=-f^{147}+f^{237}-f^{567}+f^{125}+f^{136}+f^{246}-f^{345}$, \,\,\,\,  $\{f^2,f^3,f^4,f^1,f^6,f^5,f^7\}$;\\
\item $37D1$:\,\, $f^i=e^i,  \quad 1\leq i\leq 4, \qquad  f^j=\frac{\sqrt{3}}{6}\,e^j, \quad 5\leq j\leq 7$. \\
Hence, $37D1=\left(0,0,0,0, \frac{\sqrt{3}}{6}(f^{12} - f^{34}, \frac{\sqrt{6}}{6}(f^{13}+f^{24}), \frac{\sqrt{6}}{6}(f^{14}-f^{23}\right)$, and\\
$\varphi_{7}=f^{137}+f^{247}-f^{567}+f^{125}+f^{146}-f^{236}-f^{345}$, \,\,\,\,  $\{f^1,f^3,f^2,f^4,f^5,-f^6,f^7\}$.
\end{itemize}
It is straightforward to check that each $\varphi_{i}$ is coclosed and induces the inner product $g_{i} = \sum_{j = 1}^7 (f^j)^2$. By \cite{EF}, each  $g_{i}$
is a nilsoliton on the corresponding Lie algebra, which completes the proof.
\end{proof}
\begin{remark}
Note that for each of the Lie algebras $17-37D1$ considered in the previous theorem, the change of basis
from $\{e^i\}$ to $\{f^i\}$, given in the proof of Theorem \ref{nilsoliton-indecomp}, defines an automorphism of the Lie algebra,
but it is not an isomorphism between  the $\Gtwo$-structure defined in the proof of Theorem \ref{nilsoliton-indecomp}
and the  $\Gtwo$-structure defined in the proof of Theorem \ref{indecomposable-cocalibrated}. In fact, these structures define 
different metric. 
\end{remark}

\begin{theorem}\label{nilsoliton-decomposable}
Any decomposable 2-step nilpotent Lie algebra has a  coclosed  $\Gtwo$-structure inducing a nilsoliton.
\end{theorem}
\begin{proof}
By  Corollary \ref{cor:2-stepcoc}, any  decomposable $2$-step nilpotent  Lie algebra admitting a coclosed  $\Gtwo$-structure  is  isomorphic to one of the  following  seven  Lie algebras  ${\frak n}_s$, with  $s\in\{2,3,4,5,6,7,10\}$, defined in Theorem \ref{decomposable-cocalibrated}.
For each Lie algebra $ {\frak k}_i$  consider
For each of these Lie algebras, we consider the $G_2$ form $\varphi$, and  the appropriate coframe defining $\varphi$, given by
\begin{itemize}
\item $\lie{n}_2$:\; $\varphi= -  e^{137} -  e^{247}-e^{567} + e^{126}+  e^{145}  - e^{235} - e^{346}$, \; $\{e^1,e^3,e^2,e^4,e^6,-e^5,-e^7\}$;\\
\item $\lie{n}_3$:\; $\varphi=  -  e^{127} +  e^{347} -e^{567} +  e^{135}  + e^{146} -   e^{236} + e^{245}$, \; $\{e^1,e^2,e^3,-e^4,e^5,e^6,-e^7\}$;\\
\item $\lie{n}_4$:\; $\varphi= e^{147} + e^{237}+e^{567}- e^{126} - e^{135}+ e^{245}- e^{346}$, \,\; $\{e^1,e^4,e^2,e^3,-e^6,e^5,e^7\}$;\\
\item $\lie{n}_5$:\; $\varphi= \frac{\sqrt{2}}{2} (e^{125} +  e^{126}  +  e^{135}  -   e^{136}  - e^{245} +  e^{345}  + e^{346}  +  e^{246}) + e^{147}  + e^{237} +  e^{567}$,
\begin{flushright}
$\{e^1,e^2,\frac{\sqrt{2}}{2}(e^5+e^6),e^3,\frac{\sqrt{2}}{2}(e^5-e^6),e^4,e^7\}$;
\end{flushright}
\item $\lie{n}_6$:\; $\varphi=e^{127}+e^{347}+e^{567}-e^{235} +e^{246}-e^{136}-e^{145}$, \,\; $\{e^1,e^2,e^6,-e^5,e^3,e^4,e^7\}$;\\
\item $\lie{n}_7$:\; $\varphi= e^{137} + e^{247} - e^{567} - e^{125}- e^{146}+ e^{236}+ e^{345}$, \,\; $\{e^1,e^3,e^2,e^4,-e^5,e^6,e^7\}$;\\
\item $\lie{n}_{10}$: $\varphi= - e^{147} - e^{257}- e^{367}  + e^{126}- e^{135} + e^{234} - e^{456}$, \; $\{e^1,e^4,e^3,e^6,-e^5,e^2,-e^7\}$.\\
\end{itemize}

It is straightforward to check that each form $\varphi$ is coclosed on the corresponding Lie algebra, and it defines the metric 
$g=\sum_{i=1}^7 (e^i)^2$, which is a nilsoliton as it is proved in \cite{Will}. 
\end{proof}

Finally, we show an example of a $3$-step nilpotent Lie algebra admitting  coclosed $\Gtwo$-structures and a nilsoliton which is not determined 
by any  coclosed  $\Gtwo$-structure.

\begin{example}\label{example:1}
 Consider the $3$-step nilpotent decomposable Lie algebra $\lie{n}_8$ defined in Theorem \ref{decomposable-cocalibrated}
as $\lie{n}_8=(0,0,0,0,12,14+25,0)$, with respect to a basis $\{e^1,\ldots, e^7\}$ of the dual space $\lie{n}_{8}^*$.  Clearly,
in the basis $\{f^1=e^1,f^2=\frac{\sqrt{2}}{2}\,e^2, f^3=e^3, \ldots, f^7=e^7\}$,
 the Lie algebra $\lie{n}_8$ is defined by the 
structure equations $$(0,0,0,0,\sqrt{2}f^{12},f^{14}+\sqrt{2}f^{25},0).$$ 
By \cite{Culma}, the inner product $$g = \sum_{i = 1}^7  (f^i)^2$$ is a nilsoliton on $\lie{n}_8$, and 
so it is unique up to scaling and  automorphism of the Lie algebra \cite{Lauret}. 

On the other hand, by Theorem \ref{decomposable-cocalibrated}, $\lie{n}_8$ admits a coclosed $\Gtwo$-structure.
However, we will show that the Lie algebra $\lie{n}_8$ does not carry any coclosed $\Gtwo$-structure inducing the nilsoliton $g$. For this, we consider 
the $\Gtwo$-structure defined by the 3-form 
$$\varphi_0=e^{127}+e^{347}+e^{567}+e^{135}-e^{146}-e^{236}-e^{245}.$$
Denote by $\star$ the Hodge operator defined by $\varphi_0$.  By using the formula $(3.6)$ in  \cite{Bryant}, we know that every positive three-form $\varphi$ on $\lie{n}_8$, compatible with the previous metric $g$ and orientation, can be written as 
$$\varphi=(a^2-|\alpha|^2)\varphi_0+2a*(\alpha\wedge\varphi_0)+i(\alpha\circ\alpha),$$
where $a$ is a constant, $\alpha$  is a one-form satisfying $a^2+|\alpha|^2=1$, and $i(\alpha \circ \alpha)$  is a three-form depending on $\alpha$ and $\varphi_0$. Consequently
$$*\varphi=(a^2-|\alpha|^2)*\varphi_0+2a(\alpha\wedge\varphi_0)+*i(\alpha\circ\alpha).$$
Consider a generic one form  $\alpha=\sum_{j=1}^7 \alpha_j e^j$, with  $\alpha_j\in\R$. Since     $$i(\alpha\circ\alpha)=  \sum_{1\leq j,k\leq 7}\alpha_j\alpha_ke^j\wedge(\iota_{e_k}\varphi_0),$$  we obtain the following  explicit expression  for  $*\varphi$
\begin{align*}
&*\varphi=
(a^2-|\alpha|^2)e^{3456}+(2a\alpha_{3}+\alpha_{1}\alpha_{5}-\alpha_{2}\alpha_{6}-\alpha_{4}\alpha_{7})e^{3567}+
\\&
(2a\alpha_{1}-\alpha_{2}\alpha_{7}+\alpha_{3}\alpha_{5}-\alpha_{4}\alpha_{6})e^{1347}+(a^2-|\alpha|^2)e^{1256}+(a^2-|\alpha|^2)e^{2357}+
\\&
(2a\alpha_{2}+\alpha_{1}\alpha_{7}-\alpha_{3}\alpha_{6}-\alpha_{4}\alpha_{5})e^{2347}+(2a\alpha_{2}+\alpha_{1}\alpha_{7}+\alpha_{3}\alpha_{6}+\alpha_{4}\alpha_{5})e^{2567}+
\\&
(a^2-|\alpha|^2)e^{1367}+(2a\alpha_{4}-\alpha_{1}\alpha_{6}-\alpha_{2}\alpha_{5}+\alpha_{3}\alpha_{7})e^{4567}+(a^2-|\alpha|^2)e^{1457}+
\\&
(2a\alpha_{1}-\alpha_{2}\alpha_{7}-\alpha_{3}\alpha_{5}+\alpha_{4}\alpha_{6})e^{1567}+(-a^2+|\alpha|^2)e^{2467}+(a^2-|\alpha|^2)e^{1234}+
\\&
(2-a\alpha_{1}-\alpha_{2}\alpha_{7}-\alpha_{3}\alpha_{5}-\alpha_{4}\alpha_{6})e^{1236}+(2-a\alpha_{1}-\alpha_{2}\alpha_{7}+\alpha_{3}\alpha_{5}+\alpha_{4}\alpha_{6})e^{1246}+
\\&
(2-a\alpha_{2}+\alpha_{1}\alpha_{7}+\alpha_{3}\alpha_{6}-\alpha_{4}\alpha_{5})e^{1235}+(2a\alpha_{3}-\alpha_{1}\alpha_{5}+\alpha_{2}\alpha_{6}-\alpha_{4}\alpha_{7})e^{1237}+
\\&
(2a\alpha_{2}-\alpha_{1}\alpha_{7}+\alpha_{3}\alpha_{6}-\alpha_{4}\alpha_{5}e^{1246}+(2a\alpha_{4}+\alpha_{1}\alpha_{6}+\alpha_{2}\alpha_{5}+\alpha_{3}\alpha_{7})e^{1247}+
\\&
(2a\alpha_{5}+\alpha_{1}\alpha_{3}-\alpha_{2}\alpha_{4}-\alpha_{6}\alpha_{7})e^{1257}+(2a\alpha_{6}-\alpha_{1}\alpha_{4}-\alpha_{2}\alpha_{3}+\alpha_{5}\alpha_{7})e^{1267}+
\\&
(2a\alpha_{4}+\alpha_{1}\alpha_{6}-\alpha_{2}\alpha_{5}-\alpha_{3}\alpha_{7}e^{1345}+(2a\alpha_{3}-\alpha_{1}\alpha_{5}-\alpha_{2}\alpha_{6}+\alpha_{4}\alpha_{7})e^{1346}+
\\&
(-2a\alpha_{6}+\alpha_{1}\alpha_{4}-\alpha_{2}\alpha_{3}+\alpha_{5}\alpha_{7})e^{1356}(-2a\alpha_{7}-\alpha_{1}\alpha_{2}-\alpha_{3}\alpha_{4}-\alpha_{5}\alpha_{6})e^{1357}+
\\&
(-2a\alpha_{5}-\alpha_{1}\alpha_{3}-\alpha_{2}\alpha_{4}-\alpha_{6}\alpha_{7})e^{1456}+(2a\alpha_{7}+\alpha_{1}\alpha_{2}-\alpha_{3}\alpha_{4}-\alpha_{5}\alpha_{6})e^{1467}+
\\&
(2a\alpha_{3}+\alpha_{1}\alpha_{5}+\alpha_{2}\alpha_{6}+\alpha_{4}\alpha_{7})e^{2345}+(-2a\alpha_{4}+\alpha_{1}\alpha_{6}-\alpha_{2}\alpha_{5}+\alpha_{3}\alpha_{7})e^{2346}+
\\&
(-2a\alpha_{5}+\alpha_{1}\alpha_{3}+\alpha_{2}\alpha_{4}-\alpha_{6}\alpha_{7})e^{2356}+(2a\alpha_{7}-\alpha_{1}\alpha_{2}+\alpha_{3}\alpha_{4}-\alpha_{5}\alpha_{6})e^{2367}+
\\&
(2a\alpha_{6}+\alpha_{1}\alpha_{4}-\alpha_{2}\alpha_{3}-\alpha_{5}\alpha_{7})e^{2456}+(2a\alpha_{7}-\alpha_{1}\alpha_{2}-\alpha_{3}\alpha_{4}+\alpha_{5}\alpha_{6})e^{2457}+
\\&
(2a\alpha_{5}-\alpha_{1}\alpha_{3}+\alpha_{2}\alpha_{4}-\alpha_{6}\alpha_{7})e^{3457}+(2a\alpha_{6}+\alpha_{1}\alpha_{4}+\alpha_{2}\alpha_{3}+\alpha_{5}\alpha_{7})e^{3467}.
\end{align*}
Therefore
\begin{align*}
&d(*\varphi)=
-(a^2-|\alpha|^2)\sqrt{2}e^{12357}-(2a\alpha_{4}-\alpha_{1}\alpha_{6}-\alpha_{2}\alpha_{5}+\alpha_{3}\alpha_{7})\sqrt{2}e^{12467}+
\\&
(-2a\alpha_{7}\sqrt{2}-\sqrt{2}\alpha_{1}\alpha_{2}+\sqrt{2}\alpha_{3}\alpha_{4}+\sqrt{2}\alpha_{5}\alpha_{6}-2a\alpha_{2}-\alpha_{1}\alpha_{7}-\alpha_{3}\alpha_{6}-\alpha_{4}\alpha_{5})e^{12457}+
\\&
(-2a\alpha_{3}\sqrt{2}+\sqrt{2}\alpha_{1}\alpha_{5}+\sqrt{2}\alpha_{2}\alpha_{6}-\sqrt{2}\alpha_{4}\alpha_{7}+2a\alpha_{5}-\alpha_{1}\alpha_{3}-\alpha_{2}\alpha_{4}+\alpha_{6}\alpha_{7})e^{12345}+
\\&
-(2a\alpha_{3}+\alpha_{1}\alpha_{5}-\alpha_{2}\alpha_{6}-\alpha_{4}\alpha_{7})\sqrt{2}e^{12367}+(a^2-|\alpha|^2)\sqrt{2}e^{12346}+
\\&
(2a\alpha_{5}\sqrt{2}-\sqrt{2}\alpha_{1}\alpha_{3}+\sqrt{2}\alpha_{2}\alpha_{4}-\sqrt{2}\alpha_{6}\alpha_{7}+2a\alpha_{7}-\alpha_{1}\alpha_{2}+\alpha_{3}\alpha_{4}-\alpha_{5}\alpha_{6})e^{12347}+
\\&
(2a\alpha_{6}\sqrt{2}+\sqrt{2}\alpha_{1}\alpha_{4}+\sqrt{2}\alpha_{2}\alpha_{3}+\sqrt{2}\alpha_{5}\alpha_{7})e^{23457}+(-2a\alpha_{3}-\alpha_{1}\alpha_{5}+\alpha_{2}\alpha_{6}+\alpha_{4}\alpha_{7})e^{13457}.
\end{align*}
Then we see that $d(*\varphi)=0$ is equivalent to the  following  system  
$$
\begin{array}{l}
a^2-|\alpha|^2=0,  \quad a\alpha_{4}-\alpha_{1}\alpha_{6}-\alpha_{2}\alpha_{5}+\alpha_{3}\alpha_{7} =0, \\[2pt]
2a\alpha_{3}+\alpha_{1}\alpha_{5}-\alpha_{2}\alpha_{6}-\alpha_{4}\alpha_{7} =0,\\[2pt]
2a\alpha_{6}\sqrt{2}+\sqrt{2}\alpha_{1}\alpha_{4}+\sqrt{2}\alpha_{2}\alpha_{3}+\sqrt{2}\alpha_{5}\alpha_{7}=0,\\[2pt]
 -2a\alpha_{3}\sqrt{2}+\sqrt{2}\alpha_{1}\alpha_{5}+\sqrt{2}\alpha_{2}\alpha_{6}-\sqrt{2}\alpha_{4}\alpha_{7}+2a\alpha_{5}-\alpha_{1}\alpha_{3}-\alpha_{2}\alpha_{4}+\alpha_{6}\alpha_{7}=0, \\[2pt]
2a\alpha_{5}\sqrt{2}-\sqrt{2}\alpha_{1}\alpha_{3}+\sqrt{2}\alpha_{2}\alpha_{4}-\sqrt{2}\alpha_{6}\alpha_{7}+2a\alpha_{7}-\alpha_{1}\alpha_{2}+\alpha_{3}\alpha_{4}-\alpha_{5}\alpha_{6}=0,\\[2pt]
-2a\alpha_{7}\sqrt{2}-\sqrt{2}\alpha_{1}\alpha_{2}+\sqrt{2}\alpha_{3}\alpha_{4}+\sqrt{2}\alpha_{5}\alpha_{6}-2a\alpha_{2}-\alpha_{1}\alpha_{7}-\alpha_{3}\alpha_{6}-\alpha_{4}\alpha_{5} =0,
\end{array}
$$
in the variables $a, \alpha_j$, $j = 1, \ldots, 7$. We can show that the above system has no solution. In fact, from the first equation it follows that $a^2=|\alpha|^2$, so, it is not restrictive to   suppose that  $2a=1$.  By the second, third and fourth equations we obtain
$$
\alpha_4=\alpha_1\alpha_6+\alpha_2\alpha_5-\alpha_3\alpha_7, \, 
\alpha_3=\frac{-\alpha_1\alpha_6\alpha_7-\alpha_2\alpha_5\alpha_7+\alpha_1\alpha_5-\alpha_2\alpha_6}{1+\alpha_7^2}, \, 
\alpha_6=-\alpha_5\alpha_7.
$$
By   the last  equation $\left(d(*\varphi)\right)_{12457}=0$ we get  $\alpha_5^2 = - 1$ and  so  $\lie{n}_8$ does not admit any  coclosed $\G-$structure inducing the nilsoliton.
\end{example}

\section{Contact  metric structures induced by coclosed $\Gtwo$-structures}\label{contact}

In general, the existence of a coclosed  $\Gtwo$-structure is independent of the existence of a contact form. 
We recall that   a contact metric structure on a Riemannian manifold  $(M,g)$ is a
unit length vector field $\xi$ such that the endomorphism  $\phi$ defined by $g(\phi  \cdot , \cdot )  = \frac 12  d \eta (\cdot, \cdot)$ and the $1$-form $\eta = \langle  \xi, \cdot \rangle$  are related by
$$
\phi^2 =  - {\rm Id} + \eta \otimes \xi.
$$
A contact metric structure $(M, g, \xi, \phi)$ is called  K-contact if   the Reeb vector field $\xi$  is Killing and it  is said to be normal if the induced almost complex structure J on the product manifold  $M \times \R$ is integrable.  A Sasakian manifold.
is a  normal contact metric manifold.

 A natural  question is whether on a $7$-manifold  there exists a contact metric structure such that the  metric is induced by a  coclosed  $\Gtwo$-structure.   For 2-step nilmanifolds  we can show the following 

\begin{proposition} \label{prop:contact}
If a  $2$-step nilmanifold  $\Gamma \backslash G$  admits an invariant  contact metric structure  $(g,  \xi, \phi)$ such that   the metric $g$  is induced by a  coclosed  $\Gtwo$-structure, then $G$  is isomorphic to the  $7$-dimensional  Heisenberg Lie group.  Moreover,  $\Gamma \backslash G$   has an invariant  Sasakian structure.
\end{proposition}

\begin{proof}  Since the  involved structures are invariant, we can work at the level of the Lie algebra $\mathfrak g$ of $G$.  
In \cite{Kutsak}, Kutsak showed that a decomposable $7$-dimensional nilpotent Lie algebra
does not admit any contact structure and  the only indecomposable  $2$-step nilpotent Lie algebra admitting  a contact form is the Lie algebra (17), which is the Lie algebra of $7$-dimensional  Heisenberg Lie group. Therefore,  the result follows by Theorem \ref{indecomposable-cocalibrated}.
\end{proof}

Note that by \cite{AFV}  the Heisenberg Lie algebra (17) is indeed  the only  $7$-dimensional nilpotent Lie algebra admitting a Sasakian structure. So a natural problem is  whether on a (not Sasakian)  $7$-dimensional nilpotent Lie algebra there exists a  $K$-contact metric structure such that the  metric is induced by a  coclosed structure.

\begin{example}\label{example:2}
Consider the following indecomposable 3-step nilpotent Lie algebra
$$
 {\frak n} = (0,0,0,12,13,0,16+25+34).
$$
By \cite{Goze} we know that $\frak n$ admits  a $K-$contact metric structure defined by  $g = \sum_{i =1}^7 (e^i)^2$ and  the  Reeb vector $ \xi =e_7$.  We can  show that  $\frak g$ admits a coclosed $\Gtwo$-structure $\varphi$   whose  underlying metric  coincides with $g$.  Indeed, consider the $3$-form
$$
\varphi=-e^{167}-e^{237}+e^{457}-e^{124}-e^{135}-e^{256}+e^{346}.
$$
One can check that $g_{\varphi} = g$ and that  the Hodge  dual  $\star\varphi$ of $\varphi$ given by
$$
\star\varphi = e^{1236} - e^{1456} - e^{1257} + e^{1347} + e^{2467} + e^{2357}
$$
is closed.
\end{example}

{\small

\vspace{0.15cm}

\noindent{\sf L. Bagaglini:} 
{Dipartimento di Matematica e Informatica \lq \lq Ulisse Dini\rq \rq \\ Universit\`a di Firenze\\
Viale Giovan Battista Morgagni, 67/A \\
50134 Firenze,  Italy}\\
\email{leonardo.bagaglini@gmail.com}

\vspace{0.15cm}

\noindent{\sf M. Fern\'andez:} 
{Departamento de Matem\'aticas,
Facultad de Ciencia y Tecnolog\'{\i}a, \\
Universidad del Pa\'{\i}s Vasco, Apartado 644,
48080 Bilbao, Spain. }\\
\email{marisa.fernandez@ehu.es}

\vspace{0.15cm}

\noindent{\sf A. Fino:} {Dipartimento di Matematica \lq\lq Giuseppe Peano\rq\rq \\ Universit\`a di Torino\\
Via Carlo Alberto 10\\
10123 Torino, Italy}\\
 \email{annamaria.fino@unito.it}
\end{document}